\documentclass[oneside,a4paper,12pt,english]{amsart}
\usepackage{preamble}

\author{Márton~Elekes}
\address{\normalfont (ME) Alfréd Rényi Institute of Mathematics, Budapest, Hungary AND Eötvös Loránd University, Budapest, Hungary}
\email{elekes.marton@renyi.hu}

\author{Tamás~Kátay}
\address{\normalfont (TK) Alfréd Rényi Institute of Mathematics, Budapest, Hungary}
\email{13heted@gmail.com}

\author{Anett~Kocsis}
\address{\normalfont (AK) Eötvös Loránd University, Budapest, Hungary}
\email{sakkboszi@gmail.com}

\title{Elusive properties of countably infinite graphs}

\date{\today}

\begin{document}


\begin{abstract}
A graph property is \emph{elusive} (or evasive) if any algorithm testing it by asking questions of the form ``\emph{Is there an edge between vertices $x$ and $y$?}'' must, in the worst case, examine all pairs of vertices. Elusiveness for \emph{infinite vertex sets} has been first studied by Csernák and Soukup \cite{CSERNAK_SOUKUP_2023}, who proved that the long-standing Aanderaa--Karp--Rosenberg Conjecture --- which states that every nontrivial monotone graph property is elusive --- fails for infinite vertex sets. We extend their work by giving a closer look to the case when the vertex set is countably infinite and the ``algorithm'' terminates after infinitely many steps. Among others, we prove that connectedness is elusive, which strengthens a result of Csernák and Soukup. We give counterexamples to the infinite version of the Aanderaa--Karp--Rosenberg Conjecture even if the ``algorithm'' is required to terminate after infinitely many steps, which strengthens results~in~\cite{CSERNAK_SOUKUP_2023}.
\end{abstract}

\maketitle

\tableofcontents

\newpage

\section{Introduction}\label{s.intro}

\subsubsection*{\textbf{1.1. History}}

We start with a classical problem. Fix a vertex set $V$ and a graph property $\cals$. Two players, Seeker and Hider, play a game as follows: in each turn, Seeker chooses a pair $\{u,v\}$ of vertices from $V$ not chosen before, and Hider decides whether it becomes an edge or a not. This way they are building a graph on $V$. Seeker's goal is to decide whether the graph they are building has the property $\cals$ \emph{without playing all the pairs} from $V$. The property $\cals$ is \emph{elusive} if Hider has a winning strategy in this game.

For finite vertex sets $V$, elusive properties were introduced by Rosenberg \cite{ROSENBERG_1973} in 1973 and have been extensively studied ever since, see \cite{RIVEST_VUILLEMIN_1975, LENSTRA_BEST_VAN_EMDE_BOAS, BOLLOBAS_1976, KAHN_SAKS_STURTEVANT_1984, YAO_1988, CHAKRABARTI_KHOT_SHI, BABAI_BANERJEE_KULKARNI_NAIK_2010, KULKARNI_2013, SCHEIDWILER_TRIESCH_2013, MILLER_2013, SHPARLINSKI_2014, ANGEL_BORJA_2019}. Notably, this game is the subject of the innocent-looking yet unresolved Aanderaa--Karp--Rosenberg Conjecture, which has been open since the 70s.

\begin{con}\label{conj.AKR}
Every nontrivial monotone graph property is elusive.
\end{con}

Nontrivial means that it is neither the tautologically true nor the tautologically false property, and monotone means that it is invariant under adding further edges. Elusive properties for infinite vertex sets have been first studied by Csernák and Soukup \cite{CSERNAK_SOUKUP_2023}. Among many other interesting results, they showed that the Aanderaa--Karp--Rosenberg Conjecture fails for infinite vertex sets. Their generalization of the original problem allows the players to play transfinitely.

\subsubsection*{\textbf{1.2. Results}}

We study the case when $V=\nat$ and the game is \emph{defined} to terminate after infinitely many turns. (Thus the game is infinite but not transfinite.) We say that a property $\cals$ is \textbf{$\infty$-elusive} if Seeker does not have a winning strategy in the associated infinite game $G(\cals)$. For a formal definition, see Section~\ref{s.prelim}. While it is clear from the definitions that elusiveness implies $\infty$-elusiveness if $V=\nat$, \emph{a priori} it is unclear whether the converse holds. We prove that it fails.

\newtheorem*{restated.t.separating_notions}{Theorem \ref{t.separating_notions}}
\begin{restated.t.separating_notions}
There exists a graph property $\cals_0$ that is $\infty$-elusive but not elusive.
\end{restated.t.separating_notions}

In Section~\ref{s.elusive}, we prove that connectedness is elusive for $V=\nat$ (see Theorem~\ref{t.connectedness} and Remark~\ref{r.connectedness_transfinite}), which strengthens \cite[Thm~4.16]{CSERNAK_SOUKUP_2023}. We also prove that bipartiteness is $\infty$-elusive, which settles Problem~(1)/(iii) from \cite[Sec~3]{CSERNAK_SOUKUP_2023} for infinite but not transfinite games on $\nat$. In Theorem~\ref{t.independent_edges}, we answer a question of Csernák and Soukup asked via personal communication. We also investigate determinacy and obtain partial answers to Problem~(2) from \cite[Sec~3]{CSERNAK_SOUKUP_2023}.

Since non-$\infty$-elusiveness is strictly stronger than non-elusiveness by Theorem~\ref{t.separating_notions}, the existence of a monotone not $\infty$-elusive property \emph{witnesses the failure} of the Aanderaa--Karp--Rosenberg Conjecture for infinite vertex sets \emph{in a very strong sense}. We provide two examples of such properties.

\subsubsection*{\textbf{1.3. Paper outline}} In Section~2, we provide a detailed introduction to the game outlined above, along with other relevant definitions and basic observations. In Sections~3 and 4, we consider several natural graph properties and prove that they are $\infty$-elusive / not $\infty$-elusive. Section~5 is devoted to Theorem~\ref{t.separating_notions}, that is, separating the notions of elusiveness and $\infty$-elusiveness. In Section~6, we lay the foundations of studying $\infty$-elusive graph properties via descriptive set theory. In particular, we address the question of determinacy. We conclude the paper by listing open problems in Section~7.

For the sake of transparency, we list here all the properties for which we decided whether they are $\infty$-elusive.

\noindent
\begin{minipage}[t]{0.48\textwidth}
\begin{center}
\textbf{$\infty$-elusive}
\end{center}
\vspace{0.1cm}
\begin{itemize}
    \item Connected
    \item Bipartite
    \item Contains a cycle of length $k$\\ (for any fixed $k\geq 3$)
    \item Has diameter $\leq k$\\
    (for any fixed $2\leq k<\infty$)
    \item Contains a vertex of degree $\geq d$\\
    (for any fixed $1\leq d<\infty$)
    \item Property $\cals_0$ (see Section~\ref{s.transfinite})
    \item Sensitive properties\\
    (see Definition~\ref{d.sensitive})
\end{itemize}
\end{minipage}
\hfill 
\begin{minipage}[t]{0.48\textwidth}
\begin{center}
\textbf{Not $\infty$-elusive}
\end{center}
\vspace{0.1cm}
\begin{itemize}
    \item Contains no isolated vertex
    \item Contains $\geq k$ independent edges\\ (for any fixed $2\leq k<\infty$)
    \item Properties invariant under finite changes (see Definition~\ref{d.E_0_invariant})
\end{itemize}
\end{minipage}


\section{Preliminaries}\label{s.prelim}

In this paper, we study properties of countably infinite graphs. For our purposes, it suffices to fix $\nat$ as the common vertex set.

\begin{defi}\label{d.graph_property}
A \textbf{graph property} is an isomorphism-invariant subset $\cals$ of the set of all graphs on the vertex set $\nat$. (Isomorphism-invariant means that if $G$ and $G'$ are isomorphic graphs on $\nat$ and $G\in\cals$, then $G'\in\cals$.)
\end{defi}

\subsubsection*{\textbf{2.1. The formal setting}}\label{ss.setting}

Since we would like to make the set of graphs into a topological space, we need to formalize the notion of a graph as follows. As we only consider graphs on the vertex set $\nat$, for us, a graph $G$ is simply a subset of the edge set $\pairs$ of the complete graph on $\nat$, that is, $G \subseteq \pairs = \{ \{n, k\} : \ n, k \in \nat,\ n \neq k\}$. Also, we can identify each graph $G \subseteq \pairs$ with its characteristic function $\chi_G : \pairs \to \{0, 1\}$, which allows us to view $2^\pairs = \{ f \mid f : \pairs \to \{0,1\}\}$ as \textbf{the space of (countably infinite) graphs}. Here we consider $2 = \{0,1\}$ with the discrete topology, and $2^\pairs$ carries the product of the discrete topologies, with which it is known to be homeomorphic to the usual ternary Cantor set.

Thus it makes sense to talk about Borel graph properties, for example. We proceed by giving a more formal definition of the game outlined in the introduction. Let us fix a graph property $\cals$. Two players, Seeker and Hider play the following game $G(\cals)$, in which the positions are edge-colorings of the complete graph $(\nat,\pairs)$ with 3 colors --- red, white, and green. The game starts with the constant white coloring. In each turn, Seeker picks a white edge from $\pairs$ and Hider colors it green or red. The game terminates after infinitely many steps. Turns are numbered with natural numbers. For any $n\in\nat\cup\{\infty\}$, let $R_n$, $W_n$, and $G_n$ denote the sets of red, white, and green edges respectively after $n$ turns.
In particular, $R_\infty$, $W_\infty$, and $G_\infty$ are the color classes at the end of the game. For any $n\in\nat\cup\{\infty\}$, let $\calp_n=\{G\subseteq \pairs:\ G_n\subseteq G\subseteq G_n\cup W_n\}$. We say that Seeker \textbf{can decide the property $\cals$ after $n$ turns} if either $\calp_n\subseteq\cals$ or $\calp_n\subseteq2^\pairs\setminus\cals$. Seeker wins a run of the game if and only if $W_\infty\neq\emptyset$ and she can decide the property $\cals$ after infinitely many turns. Note that $G(\cals)=G(2^\pairs\setminus\cals)$ for any graph property $\cals$.

\begin{defi}\label{d.omega_elusive}
The graph property $\cals$ is \textbf{$\infty$-elusive} if Seeker does not have a winning strategy in $G(\cals)$.
\end{defi}

Following Csernák and Soukup \cite{CSERNAK_SOUKUP_2023}, we say that $\cals$ is \textbf{strongly $\infty$-elusive} if Hider has a winning strategy in $G(\cals)$. To the best of our knowledge, it is unknown whether every $\infty$-elusive property is strongly $\infty$-elusive, that is, whether every game of the form $G(\cals)$ is determined. For partial results, see Section~\ref{s.determinacy_and_descriptive_complexity}.

Recall that while Csernák and Soukup \cite{CSERNAK_SOUKUP_2023} worked in a very general setting, using arbitrary vertex sets and transfinite games, we restrict our attention to the vertex set $\nat$ and games that terminate after infinitely many turns. Let us say a few words about why this variant is worth studying. First, there are a large variety of problems that are challenging without adding further layers of generality. Second, one may argue that requiring the game to end after $|V|$ turns is a more natural generalization of the finite game because this does not allow Seeker to discover $|V|$-many edges before the end of the game. Third, games of infinite but not transfinite length are more natural from the perspective of descriptive set theory, and their richer theory enables us to use powerful tools such as Borel determinacy (see Section~\ref{s.determinacy_and_descriptive_complexity}).

\subsubsection*{\textbf{2.2. Basic observations and definitions}}

\begin{defi}\label{d.monotone}
A graph property $\cals$ is \textbf{monotone} if the following holds: whenever $\cals$ holds for a graph $G$ on $\nat$, it also holds for any graph on $\nat$ obtained by adding edges to $G$. Formally, $\cals$ is monotone if for every $G\in\cals$ and $G'\in 2^\pairs$ we have $G\subseteq G'\implies G'\in\cals$.
\end{defi}

As we have mentioned in the introduction, the Aanderaa--Karp--Rosenberg Conjecture fails for infinite vertex sets. However, one may still expect monotone properties to be somewhat easier to understand in general, as is illustrated by the following remark and Theorem~\ref{t.monotone_borel}, which should be contrasted with Theorem~\ref{t.complete_coanalytic}.

\begin{remark}\label{r.monotone}
If $\cals$ is monotone, then it is simpler to describe when Seeker can decide $\cals$: for any $n\in\nat\cup\{\infty\}$, Seeker can decide $\cals$ after $n$ turns if and only if either $G_n\in\cals$ or $G_n\cup W_n\in 2^\pairs\setminus\cals$.
\end{remark}

\begin{notation}\label{n.notation_list}
Let us introduce some notation and terminology.
\begin{itemize}
    \item[(a)] For any vertex $x\in\nat$ and graph $H\subseteq\pairs$, let $\deg_H(x)$ denote the degree of $x$ in $H$. If we refer to a specific turn of the game and it is understood to which one we refer, then we may drop the additional index and write simply $\deg_G(x)$ (resp.~$\deg_R(x)$, $\deg_W(x)$) instead of $\deg_{G_n}(x)$ (resp.~$\deg_{R_n}(x)$, $\deg_{W_n}(x)$). If it is convenient, we may refer to $\deg_G(x)$ (resp.~$\deg_R(x)$, $\deg_W(x)$) as the \textbf{green degree} (resp.~red degree, white degree) of $x$.
    \item[(b)] For any graph $G$ on $\nat$, note that $\bigcup G$ is the set of vertices covered by $G$. Let $G'$ denote the graph $(\bigcup G,G)$. In other words, $G'$ is derived from $G$ by discarding isolated vertices.
    \item[(c)] Let us fix an enumeration $\pairs=\{e_0,e_1,\ldots\}$, which we use throughout the paper.
    \item[(d)] For simplicity, we write $xy$ for the edge $\{x,y\}\in \pairs$ if it causes no confusion.
    \item[(e)] By a \textbf{green-white path} (resp. cycle, triangle, etc.) we mean a path (resp. graph, cycle, triangle, etc.) all of whose edges are green or white. We use the analogous terms for other (sets of) colors.
\end{itemize}
\end{notation}

\begin{remark}\label{r.minimal_edge_strategy}
Many of Hider's winning strategies have the following general structure. For every $i\in\nat$, Hider follows a (sub)strategy $\sigma_i$ as long as some condition $P_i$ is not satisfied. When $P_i$ is satisfied, Hider starts following $\sigma_{i+1}$. The strategies $\sigma_i$ are designed so that for every $i\in\nat$ if the game runs along $\sigma_i$ after finitely many steps, then Hider wins. Also, if infinitely many of the $P_i$ are satisfied, Hider wins. We call these \textbf{$\infty$-stage} strategies. The sequence of turns during which Hider follows a given $\sigma_i$ are the \textbf{stages}. Stages are numbered with natural numbers. Typical simple instances of $\infty$-stage strategies are the following, which we call \textbf{minimal-edge} strategies. At the beginning of Stage $i$, if $e_{n_i}$ is the least-indexed white edge, then we set the condition $P_i$ to be ``$e_{n_i}$ is colored green or red''. Clearly, Hider wins if she terminates infinitely many stages, hence it suffices to find strategies $\sigma_i$ that are winning provided that $e_{n_i}$ remains white.
\end{remark}

Let us briefly mention two types of properties for which one of the players has a trivial winning strategy in the associated game.

\begin{defi}\label{d.sensitive}
We call the graph property $\cals$ \textbf{sensitive} if there is a graph $G$ on $\nat$ such that either

(A) $G\in\cals$ and $H\notin\cals$ for every graph $H$ that differs from $G$ by a single edge or

(B) $G\notin\cals$ and $H\in\cals$ for every graph $H$ that differs from $G$ by a single edge.

We call $G$ a \textbf{witness} to the sensitivity of $\cals$.
\end{defi}

\begin{remark}\label{r.monotone_sensitive}
Monotone properties are not sensitive except the trivial ``being the complete graph on $\nat$'' property. Examples of sensitive properties include ``being an infinite path on $\nat$'' and ``being a tree on $\nat$'', both of which are witnessed by any graph satisfying them.
\end{remark}

\begin{remark}\label{r.sensitive_elusive}
Sensitive properties are $\infty$-elusive (and even elusive) because any graph $G$ that witnesses the sensitivity gives a simple winning strategy for Hider: she responds green exactly to the edges of $G$. It is easy to check that this strategy is winning for Hider.
\end{remark}

\begin{defi}\label{d.E_0_invariant}
A graph property $\cals$ is \textbf{invariant under finite changes} if for any graphs $G, H\in 2^\pairs$ that differ by finitely many edges we have $G\in\cals\iff H\in\cals$. 

\end{defi}

\begin{remark}\label{r.invariant_non_elusive}
Properties invariant under finite changes are not $\infty$-elusive since any strategy that keeps exactly one edge white is winning for Seeker.
\end{remark}

Although we will not use it explicitly, some readers may find the following perspective interesting.

\begin{remark}\label{r.graph_of_graphs}
Consider the graph $\graphs$ on the vertex set $2^\pairs$ defined by
$$\{G,H\}\in\graphs\iff\text{$G$ and $H$ differ by a single edge}.$$
This is a bipartite graph of cardinality continuum with countably infinite connected components. Note that a graph property $\cals\subseteq 2^\pairs$ is
\begin{itemize}
    \item invariant under finite changes if and only if it is a union of connected components of $\graphs$,
    \item sensitive if and only if either $\graphs|_\cals$ or $\graphs|_{2^\pairs\setminus\cals}$ has an isolated vertex (where $\graphs|_\cala$ denotes the induced subgraph of $\graphs$ on the vertex set $\cala\subseteq 2^\pairs).$
\end{itemize}
\end{remark}

\subsubsection*{\textbf{2.3. Descriptive set theory}}\label{ss.descriptive_set_theory}

In this subsection, we recall well-known facts, definitions, and notation from descriptive set theory. For further background, see \cite{KECHRIS}. Since only Proposition~\ref{p.S_0_not_borel} and Section~\ref{s.determinacy_and_descriptive_complexity} involve descriptive set theory, readers only interested in the purely combinatorial results may skip Subsections~2.3 and 2.4.

Let $A$ be a countable set. Let $A^{<\omega}=\bigcup_{n\in\nat} A^n$, and for any $s\in A^{<\omega}$ let $l(s)$ denote the length of $s$. For any $x\in A^\nat$, let $x|_n=(x(0),\ldots,x(n-1))$. For any $s\in A^{<\omega}$ and $x\in A^\nat\cup A^{<\omega}$, we write $s\subseteq x$ if $x$ extends $s$ as a function. For any $s\in A^{<\omega}$ and $a\in A$, let $s^{\frown}a=s\cup\{(l(s),a)\}$.

For any countably infinite set $B$, the set $A^B$ of all functions from $B$ to $A$ naturally carries the product of the discrete topologies, with which it is homeomorphic to the Cantor set $2^\nat$ if $A$ is finite and to the Baire space $\nat^\nat$ if $A$ is infinite. Both are Polish spaces and in either case, for any fixed enumeration $\{b_0,b_1,\ldots\}$ of $B$, clopen sets of the form
$$N_s=\{x\in A^B:\ \forall i<l(s)\ (x(b_i)=s(i))\}$$
with $s\in A^{<\omega}$ constitute a basis in $A^B$.

A \textbf{tree} on the set $A$ is a set $T\subseteq A^{<\omega}$ that is closed under taking initial segments. A tree $T\subseteq A^{<\omega}$ is \textbf{pruned} if for every $s\in T$ there is $a\in A$ such that $s^{\frown}a\in T$. The \textbf{body} of a tree $T$ is the set $[T]=\{x\in A^\nat:\ \forall n\in\nat\ (x|_n\in T)\}$ of all infinite branches of $T$. The body of any tree is a closed subset of $A^\nat$. In particular, it is a Polish space with the subspace topology. Note that $\trees=\{T\subseteq A^{<\omega}:\ T \text{ is a tree}\}$ is a closed subspace of $2^{A^{<\omega}}$, hence it is Polish.

We can describe 2-player infinite games with trees as follows. Let $A$ be the set of all possible moves during the game. We call the elements of $A^{<\omega}$ the \textbf{positions} of the game. The rules of the game can be represented by a pruned tree $T\subseteq A^{<\omega}$, which consists of the \textbf{legal positions} --- positions that can be reached following the rules. Then $[T]$ is the set of all possible \textbf{runs} of the game. To complete the definition of the game, we need to specify a \textbf{payoff set} $W\subseteq [T]$. We say that Player I wins the run $x\in [T]$ if and only if $x\in W$. Otherwise, Player II wins. We denote this game by $G(T,W)$. If the rules are understood, we may write $G(W)$. For a family $\Gamma$ of sets in Polish spaces, we call $G(W)$ a \textbf{$\Gamma$-game} if $W\in\Gamma([T])$. For example, $G(W)$ is a Borel game if $W$ is a Borel set in $[T]$.

For any $n\in\nat^+$, a subset $A$ of a Polish space $X$ is
\begin{itemize}
    \item $\mathbf{\Sigma^0_1}$ if it is open,
    \item $\mathbf{\Pi^0_n}$ if its complement is $\mathbf{\Sigma^0_n}$,
    \item $\mathbf{\Sigma^0_{n+1}}$ if it is a countable union of $\mathbf{\Pi^0_n}$ sets,
    \item $\mathbf{\Delta^0_n}$ if it is $\mathbf{\Sigma^0_n}$ and $\mathbf{\Pi^0_n}$,
    \item \textbf{analytic} if $A=f(Y)$ for some Polish space $Y$ and continuous map $f:Y\to X$,
    \item \textbf{co-analytic} if $X\setminus A$ is analytic,
    \item \textbf{complete analytic (resp.~complete co-analytic)} if for every analytic (resp.~co-analytic) set $B\subseteq\nat^\nat$ there is a continuous map $f:\nat^\nat\to X$ such that $B=f^{-1}(A)$.
\end{itemize}
It is well-known that every Borel set is analytic and co-analytic. Complete analytic sets and complete co-analytic sets are not Borel. A well-known example of a complete co-analytic set is the set of well-founded trees $\wellfounded=\{T\in\trees:\ [T]=\emptyset\}$, viewed as a subset of $\trees$. For technical reasons, we will need that $\wellfounded\setminus\{\emptyset,\{\emptyset\}\}$ is also complete co-analytic, which is witnessed by the continuous map $f:\trees\to\trees, T\mapsto \{\emptyset,(0)\}\cup\{(0,1)^{\frown} s:\ s\in T\}$ since $\wellfounded=f^{-1}(\wellfounded\setminus\{\emptyset,\{\emptyset\}\})$.

\subsubsection*{\textbf{2.4. Spaces of compact sets}}

For a topological space $X$, let $\calk(X)$ denote the set of all nonempty compact subsets of $X$. The set $\calk(X)$ is equipped with the topology generated by sets of the form
$$\{K\in\calk(X):\ K\cap U\neq\emptyset\}\quad\text{ and }\quad \{K\in\calk(X):\ K\subseteq U\}$$
with $U\subseteq X$ open.
This is called the \textbf{Vietoris topology}. It is well-known that $\calk(X)$, which is sometimes called the \textbf{hyperspace} of $X$, inherits many topological properties of $X$. For example, if $X$ is Polish, then $\calk(X)$ is also Polish. For technical reasons, we will need the following simple proposition.

\begin{prop}\label{p.vietoris_subbasis}
Sets of the form
\begin{equation}\label{eq.vietoris_subbasis}
\{K\in\calk(2^\nat):\ K\cap N_s\neq\emptyset\}\quad\text{ and }\quad\{K\in\calk(2^\nat):\ K\cap N_s=\emptyset\}\tag{$\star$}
\end{equation}
with $s\in 2^{<\omega}$ constitute a subbasis for the Vietoris topology on $\calk(2^\nat)$.
\end{prop}

\begin{proof}
Clearly, these sets are open. To check that they generate the topology fix any open sets $U, V\subseteq X$ and $L\in\{K\in\calk(X):\ (K\cap U\neq\emptyset)\land (K\subseteq V)\}$. Since sets of the form $N_s$ with $s\in 2^{<\omega}$ form a basis in $2^\nat$, there is $n\in\nat$ and $t_0,\ldots,t_n\in 2^{<\omega}$ such that $N_{t_0}\subseteq U$ and $N_{t_0}\cap L\neq\emptyset$ and $2^\nat\setminus V\subseteq\bigcup_{i=1}^{n} N_{t_i}\subseteq 2^\nat\setminus L$. Thus
$$L\in\{K\in\calk(2^\nat):\ K\cap N_{t_0}\neq\emptyset\}\subseteq \{K\in\calk(2^\nat):\ K\cap U\neq\emptyset\}$$
and
$$L\in\bigcap_{i=1}^n \{K\in\calk(2^\nat):\ K\cap N_{t_i}=\emptyset\}\subseteq\{K\in\calk(2^\nat)\ K\subseteq V\},$$
which shows that the topology generated by sets of the form (\ref{eq.vietoris_subbasis}) is not weaker than the Vietoris topology.
\end{proof}

\section{\texorpdfstring{$\infty$}{omega}-elusive properties}\label{s.elusive}

In this section, we prove the $\infty$-elusiveness of several natural graph properties. As a warm-up, we start with a theorem that admits a simple proof via a minimal-edge strategy.

\begin{theorem}\label{t.k_cycle}
The property ``$G$ contains a cycle of length $k$'' is $\infty$-elusive for every $k\geq 3$.
\end{theorem}

\begin{proof}
We describe a minimal-edge winning strategy for Hider (recall Remark~\ref{r.minimal_edge_strategy}).

At the beginning of Stage $n$, let $e_{i_n}=uv$ be the least-indexed white edge. Hider picks $k-2$ vertices $u_0,\ldots,u_{k-3}$ so that they are covered by only white edges. As long as Seeker does not play $e_{i_n}$, Hider colors an edge green if and only if it belongs to the path $p=(u,u_0,\ldots,u_{k-3},v)$. If $uv$ is played, she colors it red.

By Remark~\ref{r.minimal_edge_strategy}, it suffices to verify that for any $n\in\nat$ Hider wins if Seeker does not terminate Stage $n$ (by playing $e_{i_n}$). Note that any cycle Hider may create during Stage $n$ contains the path $p$. Since she does not color $uv$ green, such a cycle cannot be of length $k$, hence $G_\infty$ does not contain a cycle of length $k$. On the other hand, $\{uu_0,u_0u_1,\ldots,u_{k-4}u_{k-3},u_{k-3}v,vu\}\subseteq G_\infty\cup W_\infty$, thus Seeker cannot decide whether there is a cycle of length $k$ after infinitely many turns. Hider wins.
\end{proof}

\begin{remark}
The property ``$G$ is of girth at most $k$'' is also $\infty$-elusive. Hider can use exactly the same strategy as in the proof of Theorem~\ref{t.k_cycle}.
\end{remark}

Let us continue with a result that is \textbf{not} based on a minimal-edge strategy.

\begin{theorem}\label{t.connectedness}
Connectedness is an $\infty$-elusive property.
\end{theorem}

\begin{proof}
At any point of the game, for any $k\in\nat$, we say that a path is \emph{a good path to $k$} if it is an increasing finite sequence $0=v_0<v_1<\ldots<v_l=k$ of vertices which is also a green-white path.

We describe a winning strategy for Hider. At any turn of the game, if Seeker plays $nk$ with $n<k$, then Hider responds \emph{green} if and only if the answer \emph{red} would ruin all good paths to $k$. Assume Hider follows this strategy. We verify that she wins.

\textbf{Claim 1.} There exists a good path to every $m\in\nat$ in every turn.

At the beginning, there is a good path to every $m\in\nat$. Now assume that there is a good path to every $m\in\nat$ and Seeker plays $nk$ next with $n<k$. If Hider colors it green, we are done. If she colors it red and she ruins a good path $(v_0,\ldots,v_l)$ by doing so, then $n=v_i,$ and $k=v_{i+1}$ for some $i<l$, and we can replace $(v_0,\ldots,v_{i+1})$ by some good path to $k$, which exists because Hider colored $nk$ red. Thus a good path to $v_l$ still exists, which proves the claim.

\textbf{Claim 2.} It cannot happen in any turn that for some $n<k<m$ both $nm$ and $km$ are green.

By Claim 1, there are good paths to both $n$ and $k$ in every turn. Thus after coloring either of the edges $nm$ and $km$ green, Hider will not color the other one green, which proves the claim.

\textbf{Claim 3.} It cannot happen in any turn that for some $n<k$ the edge $nk$ is white and there is a green path between $n$ and $k$.

Suppose the contrary: for some $n<k$ the edge $nk$ is white and there is a green path $(v_0,v_1,\ldots,v_l)$ with $v_0=n$ and $v_l=k$. If we had $v_i<k$ for each $i<l$, then Hider would not have colored the edge $v_{l-1}k$ green since there were good paths to $k$ (through $nk$) not containing $v_{l-1}k$, this is a contradiction. On the other hand, if there is some $i<l$ such that $v_i>k$, then the existence of the largest $v_i$ together with its neighbors contradicts Claim 2. This proves the claim.

It follows from Claim 1 that even at the end of the game there is a good path to every $m\in\nat$: for any $m\in\nat$, there are only finitely many edges between the vertices $0,1,\ldots,m$, hence after finitely many turns Hider cannot ruin more good paths to $m$.

Thus $G_\infty\cup W_\infty$ is connected. Then Seeker cannot infer that the graph is disconnected. Suppose she infers that the graph is connected, which is, by Remark~\ref{r.monotone}, equivalent to $G_\infty$ being connected. Then, by Claim 3, we must have $W_\infty=\emptyset$, therefore Hider wins.
\end{proof}

\begin{remark}\label{r.connectedness_transfinite}
Observe that the above proof works even if the game is transfinite. (Claim 1 holds in limit turns by the paragraph before the last.) Thus Theorem~\ref{t.connectedness} strengthens Theorem~4.16 in \cite{CSERNAK_SOUKUP_2023} and settles (iv) of Problem (1) from Section 3 in \cite{CSERNAK_SOUKUP_2023} for countably infinite vertex sets.
\end{remark}

\begin{theorem}\label{t.bipartite}
Being bipartite is an $\infty$-elusive property.
\end{theorem}

\begin{proof}

Recall the well-known fact that a graph is bipartite if and only if it does not contain a cycle of odd length.

We describe an $\infty$-stage winning strategy for Hider. Here, Stage $n$ will start when some conditions, which imply that all the edges between the vertices $0,\ldots,n-1$ are colored green or red, are satisfied.

Naturally, Hider will play so that she never creates a green cycle of odd length. Then for every green-white cycle of odd length, Seeker must play a white edge of this cycle at some point. This is what \emph{powers} Hider's strategy.

In Turn $0$, Hider colors the very first edge green. Then Stage $0$ begins.

Let Stage $i$ start with Turn $j_i$, and let $n_i$ be the least vertex in $\nat\setminus\bigcup G_{j_i}$ (recall Notation~\ref{n.notation_list} (b) and Subsection~2.1). Assume that the following hold for $j=j_i$ and $n=n_i$.

(1) The graph ${G_j}'$ is connected and bipartite (recall Notation~\ref{n.notation_list} (b)).

(2) Every edge between the vertices of ${G_j}'$ is either green or red.

(3) We have $\{0,\ldots,n-1\}\subseteq \bigcup G_j$.

Stage $i+1$ starts in Turn $j$ if conditions (1)-(3) are satisfied with $n=n_i+1$. Clearly, it suffices to find for every $i\in\nat$ a strategy $\sigma_i$ for Hider such that if Hider follows $\sigma_i$, then either Stage $i$ terminates in finitely many turns or Hider wins (after infinitely many turns). 

Thus we will focus on the (sub)game $Z$ that starts with Turn $j_i$ and terminates (with Hider winning) if Stage $i$ terminates, otherwise the original rules apply. Put $n'=n_i$ and $j=j_i$ for notational convenience. Let $\sigma$ denote the strategy we are looking for.
Let $X$ and $Y$ denote the parts of the bipartite graph ${G_j}'$. Next, we will define three finite games $G_0^*, G_1^*, G_2^*$ on vertex sets $V_0, V_1, V_2$ respectively and we will show that if Hider has a winning strategy in each $G_s^*$, that gives us a suitable $\sigma$. Finally, we will describe a winning strategy for Hider in each $G_s^*$.

Let $V_0=\{n,x,y\}$, $V_1=\{n,k,x,y\}$, and $V_2=\{n,k,l,x,y\}$. For each $s<3$, the positions of the game $G_s^*$ are edge-colorings of the complete graph on $V_s$ with 3 colors: red, white, and green. The starting positions are the following.
\begin{itemize}
    \item $G_0^*$: $xy$ is green, $nx$ and $ny$ are white.
    \item $G_1^*$: $xy$ is green, $nx$ is red, all the other edges are white.
    \item $G_2^*$: $xy$ is green, $nx$ and $ny$ are red, all the other edges are white.
\end{itemize}
The rules are the same in each $G_s^*$. As in the infinite game, Seeker picks white edges and Hider colors them green or red. The game terminates when any of the following three conditions are satisfied.

(I) There is a green cycle of odd length. 

(II) There is no green-white cycle of odd length. 

(III) The green edges form a connected bipartite graph that covers $n$, and there is no white edge between the vertices of this graph.

In $G_s^*$, Hider wins if and only if (III) holds when $G_s^*$ terminates. Now let us assume that for each $s<3$ Hider has a winning strategy $\tau_s$ in $G_s^*$. We describe $\sigma$ and prove that it has the desired properties. Recall that we are at the beginning of Turn $j$. If out of $X$ and $Y$ there are $s$ many ($s=0,1,2$) in which every vertex is joined to $n'$ by a red edge, then Hider will rely on $\tau_s$. (In the case $s=1$, by symmetry, we may assume that every vertex in $X$ is joined to $n'$ by a red edge.) Let us consider the case $s=2$. The cases $s=0,1$ work the same way but they are simpler. Fix any vertices $k', l'\in\nat$ covered only by white edges, and let $W=X\cup Y\cup \{n',k',l'\}$. To every position $p$ in $Z$ we will assign a position $p^*$ in $G_2^*$ as follows. Let $nk$ (resp.~$nl$, $kl$) have the same color as $n'k'$ (resp.~$n'l'$, $k'l'$). If there is a green edge of the form $k'x'$ (resp.~$k'y'$, $l'x'$, $l'y'$) with $x'\in X$ (resp.~$y'\in Y$), then let $kx$ (resp.~$ky$, $lx$, $ly$) be green. If every edge of the form $k'x'$ (resp.~$k'y'$, $l'x'$, $l'y'$) is red, then let $kx$ (resp.~$ky$, $lx$, $ly$) be red. Otherwise, let $kx$ (resp.~$ky$, $lx$, $ly$) be white. 

Now assume we are at position $p$. If Seeker plays $n'k'$ (resp.~$n'l'$, $k'l'$), then Hider responds what she would respond to $nk$ (resp.~$nl$, $kl$) according to the strategy $\tau_2$ at position $p^*$ in $G_2^*$. If Seeker plays an edge of the form $k'x'$ (resp.~$k'y'$, $l'x'$, $l'y'$) with $x'\in X$ (resp.~$y'\in Y$) that is not the last white edge of this form, then Hider responds red. If Seeker plays the last edge of the form $k'x'$ (resp.~$k'y'$, $l'x'$, $l'y'$) with $x'\in X$ (resp.~$y'\in Y$), then Hider responds what she would respond to $kx$ (resp.~$ky$, $lx$, $ly$) according to the strategy $\tau_2$ in position $p^*$ in $G_2^*$. If Seeker plays an edge that has an endpoint not in $W$, then Hider responds red.

Based on the fact that ${G_j}'$ is a connected bipartite graph, we make two observations. First, Hider never creates a green cycle of odd length in the game $Z$ because that would mean she creates one in $G_2^*$, which contradicts that $\tau_2$ is winning. Second, if they are in a position $p$ such that in $p^*$ there is a green-white cycle of odd length on $V_2$, then there is a green-white cycle of odd length on $W$ in position $p$.

Thus as long as they are in a position $p$ in $Z$ such that $p^*$ is not terminal in $G_2^*$, Seeker cannot stop playing edges indefinitely between vertices in $W$ (because she would lose after infinitely many turns). Hence after finitely many turns they reach a position $p$ such that $p^*$ is terminal in $G_2^*$. Since $\tau_2$ is winning for Hider, $p^*$ satisfies condition (III). Now it follows from the previous two paragraphs and condition (III) that position $p$ satisfies conditions (1)-(3) with $n$ replaced by $n'+1$, that is, the game $Z$ terminates and Hider wins.

The case $s=0$ (resp.~$s=1$) can be handled the same way, the only difference being that we do not need an extra vertex (resp.~we need only 1 extra vertex) in addition to $X\cup Y\cup\{n'\}$.

It remains to construct the strategies $\tau_s$. Consider the following condition, which is a weak form of (III).

($\blacksquare$) The green edges form a connected bipartite graph $\wtilde G$ that covers $n$, and the endpoints of every white edge between vertices of $\wtilde G$ are connected by a green path of even length.

We claim that every position that satisfies ($\blacksquare$) is winning for Hider. Indeed, by ($\blacksquare$), every white edge between vertices of $\wtilde G$ is the \emph{only} white edge of a green-white cycle of odd length, hence Seeker has to play all of them. Hider colors each of them red. Conditions (I) and (II) are not satisfied as long as there is a white edge between some vertices of $\wtilde G$. When no more is left, (III) is satisfied, thus Hider wins.

Thus in the following constructions, as long as ($\blacksquare$) is not satisfied, it suffices to check that $\lnot$(I) remains true (that is, there is no green cycle of odd length) when Hider colors an edge green, and $\lnot$(II) remains true (that is, there is a green-white cycle of odd length) when Hider colors an edge red.

The strategy $\tau_0$ is trivial: Hider colors the very first edge green, and ($\blacksquare$) is satisfied immediately.

To construct $\tau_1$ let $A=nk$, $B=nx$, $C=ny$, $D=kx$, $E=ky$, and $F=xy$. In the starting position, $B$ is red, $F$ is green, and all the other edges are white.

\textbf{Case 1.} Seeker plays $A$ first. Then Hider colors it green, $\lnot$(I) remains true. From now on as long as $(\blacksquare)$ is not satisfied and independently of the order in which the edges are played, Hider colors exactly $E$ and the second-played edge from $\{C,D\}$ green. Conditions $\lnot$(I) and $\lnot$(II) remain true until an edge from $\{C, D, E\}$ is colored green, and then ($\blacksquare$) is satisfied.

\textbf{Case 2.} Seeker plays $C$ first. Hider colors it green, ($\blacksquare$) is satisfied.

\textbf{Case 3.} Seeker plays $D$ first. Hider colors it red, $\lnot$(II) remains true. From now on, Hider colors every edge green until $n$ is connected to $y$ by a green path, $\lnot$(I) remains true. Then ($\blacksquare$) is satisfied.

\textbf{Case 4.} Seeker plays $E$ first. Hider colors it green, $\lnot$(I) remains true. From now on, independently of the order in which the edges are played, Hider colors $D$ red (not violating $\lnot$(II)) and the first-played edge from $\{A,C\}$ green. When an edge from $\{A, C\}$ is colored green, ($\blacksquare$) is satisfied.

It remains to construct the winning strategy $\tau_2$ for Hider in $G_2^*$. Let $A=kl$, $B=nk$, $C=nl$, $D=kx$, $E=ky$, $F=lx$, $G=ly$, $H=nx$, $I=ny$, and $J=xy$. In the starting position, $J$ is green, $H$ and $I$ are red, and all the other edges are white. The construction involves checking several cases, which we have tried to make as easy as possible. First, observe that the 12 positions shown below in Figure~\ref{f.winning_positions} all satisfy ($\blacksquare$). (Black lines represent white edges.)

\begin{figure}[ht]
\caption{Winning positions}
\label{f.winning_positions}
\centering
\includegraphics[scale=0.5]{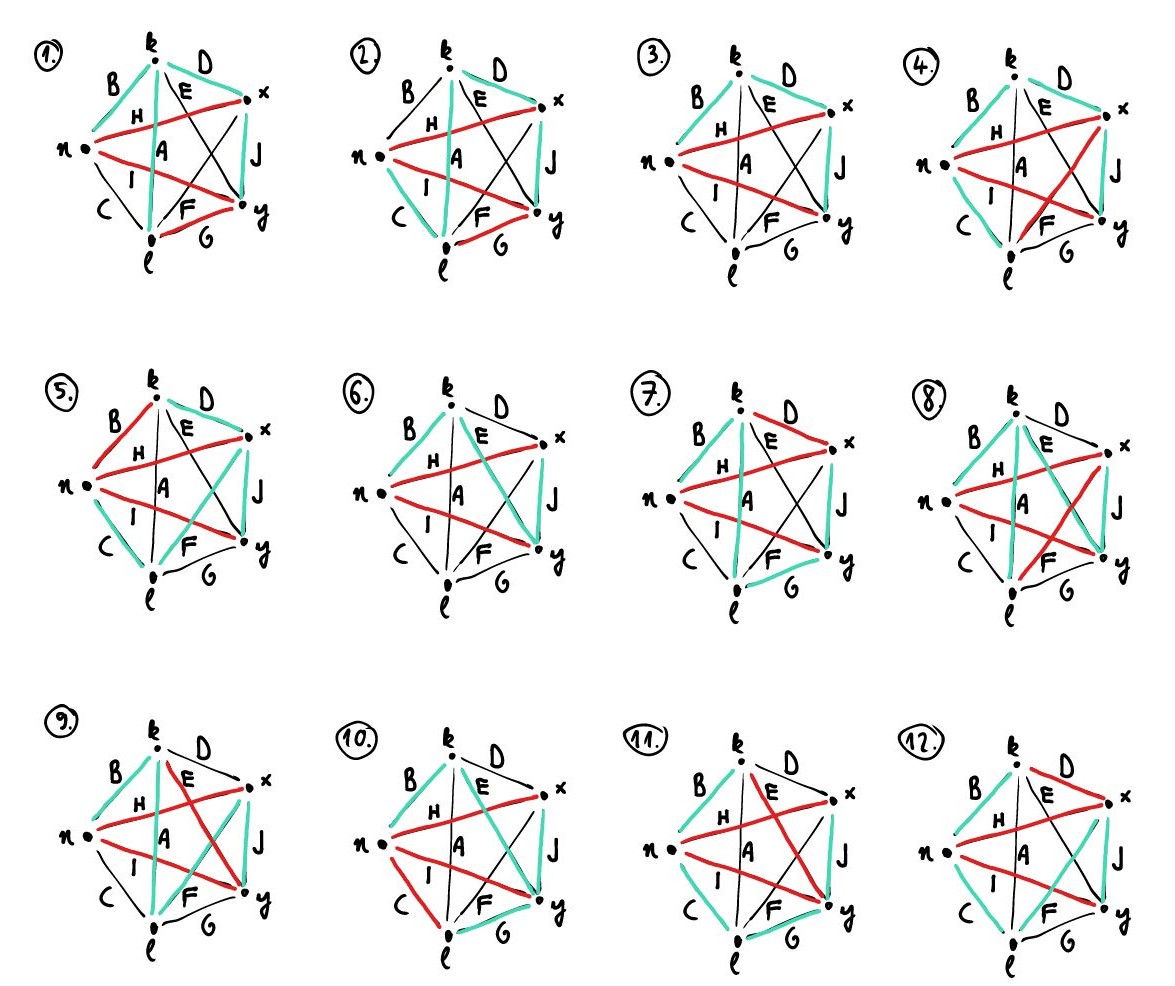}
\end{figure}

Note also that if a position satisfies ($\blacksquare$) and we change the color of some white edges to red, then ($\blacksquare$) remains true. Therefore we will say (somewhat sloppily) that we reached Position $i$ in Figure~\ref{f.winning_positions} if we reached a position that can be obtained from Position $i$ by coloring some white edges red.

The rest of the proof involves checking 21 cases, which are worked out in the Appendix.
\end{proof}

Let us conclude this section by addressing two more graph properties. Hider has an $\infty$-stage winning strategy in both cases.

\begin{theorem}
For every $1\leq d<\infty$, the property ``there exists a vertex of degree at least $d$'' is $\infty$-elusive.
\end{theorem}

\begin{proof}
If $d=1$, Hider plays red in every turn and wins. Thus, we may assume $d\geq 2$. We describe an $\infty$-stage winning strategy for Hider.

At the beginning of Stage $n$, let $x_n=\min\{x\in\nat:\ \deg_G(x)<d-1\}$. Stage $n$ terminates if $\deg_G(x_n)=d-1$. Clearly, Hider wins if she terminates infinitely many stages without creating a vertex $y$ with $\deg_G(y)\geq d$.

In Stage $n$, Hider plays as follows: she responds green exactly to the edges of the form $ax_n$ with $\deg_G(a)<d-1$. Note that as long as $\deg_G(x_n)<d-1$, Seeker must play an edge of this form at some point since she must play infinitely many edges covering $x_n$, and only finitely many of them are not of this form. Thus, Stage $n$ terminates in finitely many turns.
\end{proof}

\begin{theorem}
For any $2\leq d<\infty$, the property ``$G$ is of diameter at most $d$'' is $\infty$-elusive.
\end{theorem}

\begin{proof}
We describe a minimal-edge winning strategy for Hider.

Let us assume we are at the beginning of Stage $i$, and the least-indexed white edge is $e_{n_i}=xy$. Let $z_0,\ldots,z_{d-2},w\in\nat\setminus\{x,y\}$ be vertices covered only by white edges, and let $A=\nat\setminus \{x,y,z_0,\ldots,z_{d-2},w\}$. In Stage $i$, Hider responds green exactly to the following:
\begin{itemize}
\item the edge $yz_0$ and edges of the form $z_jz_{j+1}$ with $j\leq d-3$,
\item edges of the form $xa$ and $ya$ with $a\in A$,
\item edges of the form $bc$ with $b,c\in A\cup\{w\}$,
\item the edge $z_0w$.
\end{itemize}

\textbf{Claim 1.} During Stage $i$, there are green-white paths of length at most $d$ between every $u,v\in\nat$.

For $u,v\in A\cup\{w\}$, the statement is clear. For $u,v\in\{x,y,z_0,\ldots,z_{d-2}\}$, it is also clear. Suppose that $u\in \{x,y,z_0,\ldots,z_{d-2}\}$ and $v\in A\cup\{w\}$. If $u=z_k$ for some $k\leq d-2$, then either $(u,z_{k-1},\ldots, z_0,w)$ or $(u,z_{k-1},\ldots, z_0,w,v)$ is a green-white path of length at most $d$ from $u$ to $v$. If $u=x$ or $u=y$, then there exists a vertex $a\in A\setminus\{v\}$ such that $uav$ is a green-white path of length $2$, which proves the claim.

\textbf{Claim 2.} During Stage $i$, there is no green path of length at most $d$ between $x$ and $z_{d-2}$.

Any such green path would have to go through $z_0$, hence it suffices to check that there is no green path of length $2$ between $x$ and $z_0$. The only possible green edges covering $z_0$ are $yz_0$ and $wz_0$. However, $xy$ is white, and $xw$ is red or white, which proves the claim.

By Claims~1 and 2, if the game ends without Stage $i$ terminating, then we have $\diam(G_\infty\cup W_\infty)\leq d$ and $\diam(G_\infty)\geq d+1$, hence Hider wins.
\end{proof}

\section{Not \texorpdfstring{$\infty$}{omega}-elusive properties}\label{s.nonelusive}

In this section, we present nontrivial examples of monotone not $\infty$-elusive properties, which witness a strong failure of the infinite version of the Aanderaa--Karp--Rosenberg Conjecture.

\begin{theorem}\label{t.independent_edges}
For every $2\leq k<\infty$, the property ``$G$ contains $k$ independent edges'' is not $\infty$-elusive.
\end{theorem}

\begin{proof}
\textbf{Outline.} The main idea is the following. Seeker can win if Hider creates a vertex covered by at least $2k-1$ green edges. This puts a bound on the green-degrees, hence Seeker can win by forcing Hider to add green edges indefinitely. 

\textbf{Detailed proof.} We may assume that Hider never creates $k$ independent green edges. For any $n\in\nat$, after $n$ turns, we call an edge set $I\subseteq W_n$ \textit{irrelevant} if $G_n\cup I$ does not contain $k$ independent edges. Otherwise, we say $I$ is \textit{relevant}. An edge $e$ is irrelevant (resp.~relevant) if $\{e\}$ is irrelevant (resp.~relevant).

\textbf{Claim 1.} At any point in the game, if an edge $xy$ is white, then we claim that Seeker can force Hider to create at most $k-1$ green edges that cover every endpoint of the green edges adjacent to $xy$, except $x$ and $y$. In particular, if Hider is unable to do so, Seeker wins.

Indeed, Seeker enumerates all white edges except $xy$ and plays them in this order. Assume Hider plays in such a way that she does not create a set of at most $k-1$ green edges that covers every endpoint of the green edges adjacent to $xy$, except $x$ and $y$. At the end of the game, if $G_\infty\cup W_\infty=G_\infty\cup\{xy\}$ contains $k$ independent edges, then so does $G_\infty$ since we can replace $xy$ with one of the adjacent green edges if necessary. Thus Seeker wins, which proves Claim 1.

\textbf{Observation.} At any point in the game, if there is $x\in\nat$ with $\deg_G(x)\geq 2k-1$, then Seeker can win: she applies Claim~1 to any white edge that covers $x$.

\textbf{Claim 2.} For any $n\in\nat$, after $n$ turns, if there is $x\in\nat$ with $\deg_G(x)=1$, then Seeker can force Hider to extend $\bigcup G_n$ by coloring more edges green. Moreover, she can do so while preserving the property that there is $x\in\nat$ with $\deg_G(x)=1$.

Let $y$ be the unique vertex such that $xy\in G_n$.

\textbf{Case 1.} Every edge $xa$ with $a\in G_n\setminus\{y\}$ is red. Then, by Claim~1, Seeker can use any white edge that covers $y$ to force Hider to cover $x$ with a green edge. By doing so, Hider must color an edge that has an endpoint not in $\bigcup G_n$ green. By coloring the first such edge green, Hider extends $\bigcup G_n$ and preserves the property that there is $x\in\nat$ with $\deg_G(x)=1$.

\textbf{Case 2.} There is a white edge of the form $xa$ with $a\in G_n\setminus\{y\}$. Then Seeker plays all white edges between the vertices $(\bigcup G_n)\setminus\{x\}$ and Hider colors them. Then Seeker plays the \emph{relevant} edges of the form $xa$ with $a\in (\bigcup G_n)\setminus \{y\}$ one by one. Note that Hider must color all of them red by the definition of a relevant edge. If no white edge of the form $xa$ with $a\in(\bigcup G_n)\setminus \{y\}$ remains, then we are done by Case~1. Otherwise, a nonempty set $I$ of \emph{irrelevant} white edges of this form remains. Then, by the definition of an irrelevant edge and the fact that any two edges in $I$ are adjacent, the edge set $I$ itself is irrelevant.

Now Seeker enumerates all edges with an endpoint not in $\bigcup G_n$ and plays them in this order. Hider must color some of these edges green since otherwise $G_\infty\cup W_\infty$ could not contain $k$ independent edges by the irrelevance of $I$. By coloring the first such edge green, Hider extends $\bigcup G_n$ and preserves the property that there is $x\in\nat$ with $\deg_G(x)=1$. This proves Claim~2.

Finally, Seeker's strategy is the following. First, she forces Hider to create a green edge. (For example, by enumerating all edges except one and playing them in this order.) Then, by Claim~2, Seeker can force Hider to extend the set of green edges indefinitely. On the other hand, by the observation, Hider must keep $\deg_G(x)\leq 2k-2$ for every $x\in\nat$. It follows from a trivial estimate that a graph with at least $4k^2$ edges and maximum degree at most $2k$ contains $k$ independent edges. Thus, Seeker wins.
\end{proof}

In Theorem~\ref{t.independent_edges} the case $k=2$, which is much simpler than the general case, was proved by Csernák and Soukup, who asked the authors about the general problem via email. 

\begin{theorem}\label{t.isolated}
The property ``there is no isolated vertex'' is not $\infty$-elusive.
\end{theorem}

\begin{proof}
We describe a winning strategy for Seeker.

\textbf{Observation.} At any point in the game, if there exists a white edge $xy$ such that $x$ and $y$ are covered by green edges, then this is a winning position for Seeker since she no longer needs to play the edge $xy$.

\textbf{Claim.} For any $n\in\nat$, after $n$ turns, if there exists a vertex $x\in \nat\setminus\bigcup G_n$ such that

(A) every edge of the form $xa$ with $a\in \bigcup G_n$ is red, and

(B) for every vertex $y\notin (\bigcup G_n)\cup\{x\}$ there exists $a\in \bigcup G_n$ such that $ya$ is white,

then this is a winning position for Seeker.

Seeker enumerates all edges of the form $xy$ with $y\notin \bigcup G_n$ and plays them in this order. Hider must always respond red by (B) and the observation, hence we get that $x$ is isolated in $G_\infty\cup W_\infty$ and $W_\infty\neq\emptyset$. Seeker wins, which proves the claim.

Seeker starts by picking a vertex $x_0$. She enumerates all the edges covering $x_0$ and plays them one by one until Hider responds green, which must happen since otherwise we get that $x_0$ is isolated in $G_\infty\cup W_\infty$ and $W_\infty\neq\emptyset$.

When Hider responds green for the first time, we introduce the following labeling: let $x_1,\ldots,x_{m-1}$ denote the other endpoints of the red edges, and let $x_n$ denote the other endpoint of the first green edge.

Seeker plays the edges $x_1x_2,\ldots,x_1x_m$ in this order. By the observation, Hider must respond red to $x_1x_2,\ldots,x_1x_{m-1}$. By the claim, she must respond green to $x_1x_m$. Similarly, Seeker plays the edges $x_2x_3,\ldots,x_2x_m$ in this order, and Hider must respond red to each except the last, to which she must respond green. Seeker repeats this procedure for $x_3,\ldots,x_{m-1}$. Thus, we get that between the vertices $x_0,\ldots,x_m$ exactly the edges covering $x_m$ are green, the rest are red. Switch the labels $x_0$ and $x_m$ for convenience.

Let $y$ and $z$ be two vertices covered only by white edges. Seeker plays the edges $yz, yx_0,\ldots,yx_{m-1}$ in this order, and Hider must respond red each time since $yx_m$ is white. Let $x_{m+1},x_{m+2},\ldots$ be an enumeration of the vertex set $\nat\setminus\{y,z,x_0,\ldots,x_m\}$.

We proceed recursively. Assume that for some $k\in\nat$,
\begin{itemize}
\item the edges $x_1x_0,\ldots,x_{m+k}x_0$ are green,
\item every edge between the vertices $y,x_1,\ldots,x_{m+k-1}$ is red,
\item the edges $yz, yx_0, x_{m+k}x_1,\ldots,x_{m+k}x_{m+k-1}$ are red,
\item every other edge is white.
\end{itemize}
Then Seeker plays the edges $x_{m+k+1}x_{m+k},\ldots,x_{m+k+1}x_0$ in this order. Again, Hider must respond red to each except the last by the observation, to which she must respond green by the claim. Now Seeker plays $yx_{m+k}$, and Hider must respond red because $yx_{m+k+1}$ is white.

By repeating this procedure, Seeker forces Hider to color every edge of the form $yx_i$ red without playing any edge covering $z$. Thus, the vertex $y$ is isolated in $G_\infty\cup W_\infty$ and $W_\infty\neq\emptyset$, hence Seeker wins.
\end{proof}

\section{Elusive versus \texorpdfstring{$\infty$}{omega}-elusive}\label{s.transfinite}

In this section, we separate two notions: the notion of an elusive property and that of an $\infty$-elusive property. We do so by giving an example of a graph property $\cals_0$ such that if the game terminates after infinitely many turns, then Hider has a winning strategy, but if they are allowed to play transfinitely, then Seeker has one.

\begin{theorem}\label{t.separating_notions}
There exists a graph property $\cals_0$ that is $\infty$-elusive but not elusive.
\end{theorem}

To prove Theorem~\ref{t.separating_notions} we need the following simple lemma.

\begin{lemma}\label{l.pattern}
Let $Y=\{x\in 2^\nat:\ |\supp x|=|\nat\setminus \supp x|=\omega\}$, where $\supp x$ denotes $\{i\in\nat:\ x(i)=1\}$. There exists a function $f:Y\to \{0,1\}$ such that

(1) both $f^{-1}(0)$ and $f^{-1}(1)$ are dense in $Y$,

(2) for every $x,y\in 2^\nat$ if $|(\supp x)\Delta(\supp y)|=1$, then $f(x)\neq f(y)$.
\end{lemma}

\begin{proof}
Consider the graph $G$ on $Y$ defined by
$$\{x,y\}\in G\iff |(\supp x)\Delta(\supp y)|=1.$$
It is easy to check that it does not contain odd cycles, hence it is bipartite and admits a 2-coloring $f:Y \to \{0,1\}$. To see that both $f^{-1}(0)$ and $f^{-1}(1)$ are dense fix a basic clopen set $B\subseteq Y$ and a point $x\in B$. If $(\supp x)\Delta(\supp y)=\{n\}$ for some $y\in 2^\nat$ and large enough $n$, then $f(x)\neq f(y)$ and $x,y\in B$.
\end{proof}

\begin{proof}[Proof of Theorem~\ref{t.separating_notions}]
\textbf{Informal outline.} We construct the property $\cals_0$ to express the following. There are two finite parts $Q_0$, $Q_1$ and an infinite part $X$ of $G$ such that exactly one of the structures of $Q_0$ and $Q_1$ are irrelevant, but that which is irrelevant depends on the structure of $X$. Thus, if Seeker has already discovered $X$, then she does not need to discover one of $Q_0$ and $Q_1$. This does not help her if the game is not transfinite. The main difficulty lies in making such a property isomorphism invariant. 

\textbf{Detailed proof.} First, we define $\cals_0$. Fix a function $f$ that satisfies (1) and (2) of Lemma~\ref{l.pattern}. For convenience, we use colors again: we view edges in $G$ as green and edges in $[\nat]^2\setminus G$ as red. The property $\cals_0$ holds for $G$ if and only if the associated colored graph satisfies each of the following 10 properties (see Figure 2):

(a) It has exactly one vertex $a$ such that $\deg_G(a)=\deg_R(a)=\infty$. Let $A=\{a\}$.
    
(b) It has exactly 6 vertices, say $p_0,\ldots,p_5$ with finite red degree, and $\deg_R(p_i)=i+1$ for each $i\in\{0,\ldots,5\}$. Let $P=\{p_0,p_1\ldots,p_5\}$.
    
(c) Every edge between $A$ and $P$ is green. Every edge between vertices in $P$ is green.
    
(d) There are exactly 6 vertices $q_0,\ldots,q_5$ that are not connected to each of $p_0,\ldots,p_5$ by a green edge, and $p_iq_j$ is green if and only if $i<j$. Let $Q=\{q_0,\ldots,q_5\}$, which is disjoint from $P$ by (c).

(e) All edges between the vertex sets $A$, $Q_0=\{q_0,q_1,q_2\}$, and $Q_1=\{q_3,q_4,q_5\}$ are red.

(f) Let $X=\{x_0,x_1,\ldots\}=\nat\setminus\{a,p_0,\ldots,p_5,q_0,\ldots,q_5\}$. The green edges between $x_0,x_1,\ldots$ form an infinite path that covers every $x_i$. We may assume that $x_i$ denotes the $i$th vertex of this infinite path. (Thus $x_ix_j$ is red if $|i-j|\geq 2$.)

(g) Every edge between $X$ and $P$ is green. (This already follows from (b)+(d).)

(h) Every edge between $X$ and $Q$ is red.

(i) There are infinitely many green and infinitely many red edges between $X$ and $A$. (This follows from (a) and the fact that $\nat\setminus X$ is finite.)

(j) For the characteristic function $g$ of the set $\{i\in\nat:\ x_ia\text{ is green}\}$ the following hold. If $f(g)=0$, then every edge between $q_0,q_1,q_2$ is red. If $f(g)=1$, then every edge between $q_3,q_4,q_5$ is red.

\begin{figure}[ht]
\centering
\caption{The property $\cals_0$}
\label{f.separating_notions}
\includegraphics[scale=0.52]{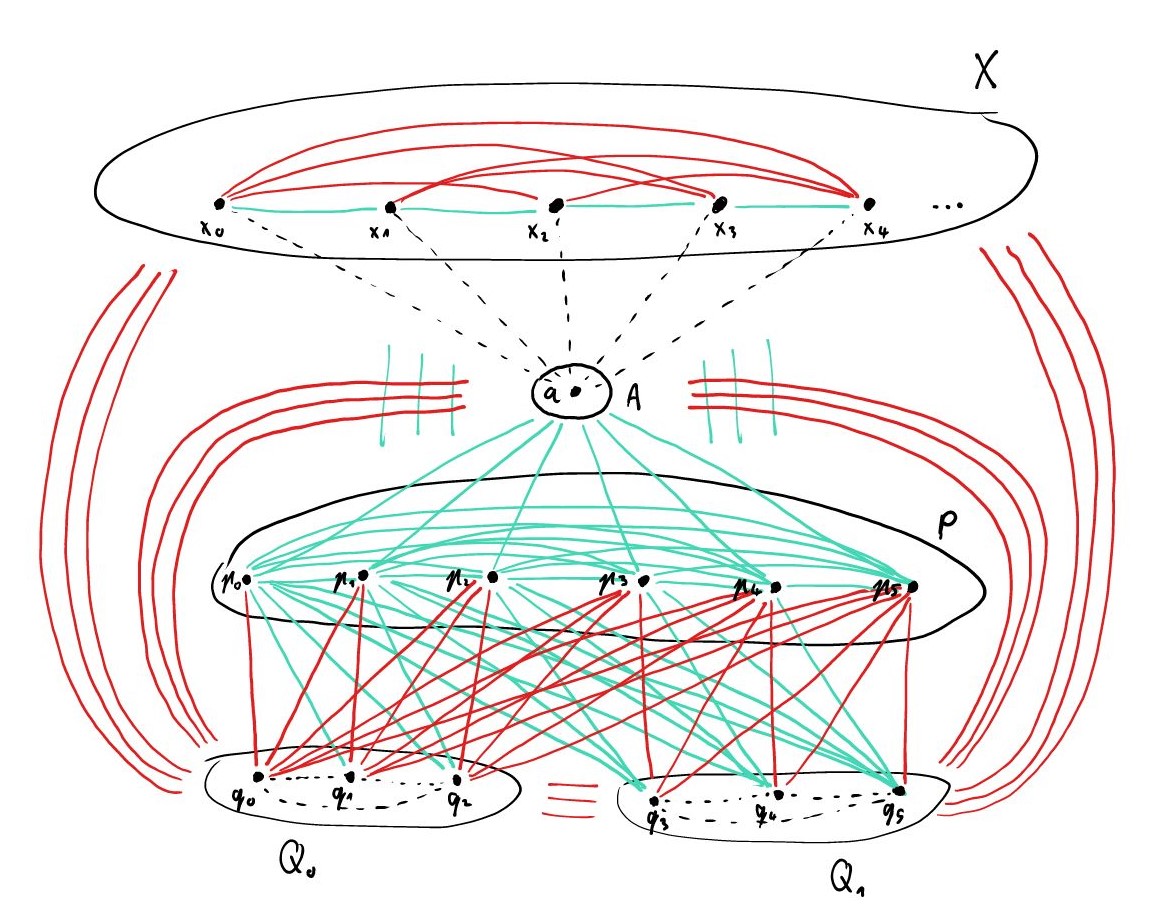}
\end{figure}

\textbf{Observation 1.} If $\cals_0$ holds for $G$, then the vertices $q_0,\ldots,q_5,x_0,x_1\ldots$ have green degree at most 9. On the other hand, $p_0,\ldots,p_5$ have red degree at most 6.

\textbf{(1) The property $\cals_0$ is not elusive.} We describe a transfinite winning strategy for Seeker. Here, in part (1) of the proof, we need to number turns with ordinals. In particular, we write $\omega$ instead of $\infty$. We divide the game into 3 stages.

\textbf{Stage 1.} As long as there is a vertex $x\in\nat$ with $\deg_G(x)+\deg_R(x)<16$, Seeker plays the edge $mn$ whose endpoints are
$$m=\min\{i\in\nat:\ \deg_G(i)+\deg_R(i)<16\},$$
and
$$n=\min\{i\in\nat\setminus\{m\}:\ \deg_G(i)+\deg_R(i)=0\}.$$
Stage 1 terminates when no such $x$ is left, which happens after $\omega$ turns.

\textbf{Observation 2.} At the end of Stage 1, the graph $G_\omega\cup R_\omega$ is a 16-regular infinite tree.

Let
$$B=\{i\in\nat:\ \text{in Stage 1, Seeker has found at least 10 green edges covering } i\}.$$
By Observation 1, if $\cals_0$ holds, then $P\subseteq B\subseteq P\cup A$. Thus, if $|B|<6$ or $|B|>7$, then Seeker wins immediately by concluding that $\cals_0$ fails. We may assume that this is not the case.

\textbf{Stage 2.} Seeker enumerates all edges covering vertices in $B$ and plays them one by one. This takes $\omega$ turns after which Stage~2 terminates. If any of the following fails, then $\cals_0$ fails and Seeker wins immediately.
\begin{itemize}
    \item There are exactly 6 edges in $B$ that are covered by infinitely many green edges and finitely many red edges.
    \item These 6 vertices can be enumerated as $p_0,\ldots,p_5$, where $p_i$ is covered by exactly $i+1$ red edges.
    \item There are exactly 6 vertices that are connected to at least one of the $p_i$ by a red edge.
    \item These 6 vertices can be enumerated as $q_0,\ldots,q_5$, where $q_j$ is connected to $p_i$ by a green edge if and only if $j>i$.
\end{itemize}

We may assume that the above hold. By Observation~2, Seeker can pick white edges $f_0\in\{q_0q_1, q_0q_2,q_1q_2\}$ and $f_1\in\{q_3q_4,q_3q_5,q_4q_5\}$.

\textbf{Stage 3.} Seeker plays every remaining white edge except $f_0$ and $f_1$. Now she can easily decide whether properties (a)-(i) hold. Suppose they hold. For deciding (j), note that since Seeker has played every edge of the form $x_ia$, she needs to play at most one of $f_0$ and $f_1$. Thus, she has won in $\omega\cdot 3+1$ turns.

\textbf{(2) The property $\cals_0$ is $\infty$-elusive.} We describe a winning strategy for Hider in $G(\cals_0)$.

We can fix a bijection between $\nat$ and $V=\{a,p_0,\ldots,p_5,q_0,\ldots,q_5,x_0,x_1,\ldots\}$, thus, for notational simplicity, we may assume that the game is played on the set $V$. Let $A=\{a\}$, $P=\{p_0,\ldots,p_5\}$, $Q=\{q_0,\ldots,q_5\}$, $Q_0=\{q_0,q_1,q_2\}$, $Q_1=\{q_3,q_4,q_5\}$, $X=\{x_0,x_1,\ldots\}$. Hider plays as follows.

Informally, whenever Seeker plays an edge that is green or red in Figure~\ref{f.separating_notions}, Hider responds according to the figure.

\label{partial_coloring} Formally, Hider responds green whenever Seeker plays any of the following edges:
\begin{itemize}
    \item edges between $A$ and $P$,
    \item edges between elements of $P$,
    \item edges of the form $\{p_i,q_j\}$ with $i<j$,
    \item edges between $P$ and $X$,
    \item edges of the form $\{x_i,x_{i+1}\}$ with $i\in\nat$.
\end{itemize}
Hider responds red whenever Seeker plays any of the following edges:
\begin{itemize}
    \item edges between $A$ and $Q$,
    \item edges of the form $\{p_i,q_j\}$ with $i\geq j$,
    \item edges between $Q_0$ and $Q_1$,
    \item edges between $Q$ and $X$,
    \item edges of the form $\{x_i,x_j\}$ with $|i-j|\geq 2$.
\end{itemize}

Regarding the rest of the edges (uncolored in the figure), Hider plays as follows. She chooses a sequence $g\in Y$ (recall Lemma~\ref{l.pattern}). As long as Seeker does not play the last edge in $[Q_0]^2$ or the last edge in $[Q_1]^2$, Hider responds green to the edge $x_ia$ if and only if $g(i)=1$, and she responds red to every edge in $[Q]^2$. 

\textbf{Claim 1.} If Seeker does not play all edges in $[Q_0]^2$ nor does she play all edges in $[Q_1]^2$, then she loses.

Observe that Seeker cannot conclude the failure of any of properties (a)-(j). Suppose she concludes that properties (a)-(i) hold. Based on properties (a)-(j), the figure, and the description of Hider's strategy it is straightforward to check the following. The unique vertex with infinite green degree and infinite red degree must be $a$. The vertices of finite red degree must be $p_0,\ldots,p_5$, their order is determined by their red degree. The vertices not connected to each of the $p_i$ by a green edge must be $q_0,\ldots,q_5$, their order is determined by the edges between $P$ and $Q$. The remaining vertices must be the $x_i$. Their order is determined by the unique infinite green path they form. In short, the vertex labels must reflect properties (a)-(i). But then she cannot decide whether (j) holds, which proves the claim.

Thus we may assume that at some point, Seeker plays the last edge in $[Q_k]^2$, where $k=0$ or $k=1$. Then Hider responds green. After this, Hider will respond red to every edge in $[Q]^2$. She picks a sequence $h\in f^{-1}(1-k)$ such that if $x_ia$ is green, then $h(i)=1$, and if $x_ia$ is red, then $h(i)=0$. This is possible because $f^{-1}(1-k)$ is dense in $Y$. From now on Hider follows $h$, that is, she responds green to the edge $x_ia$ if and only if $h(i)=1$. Regarding edges not of the form $x_ia$ and not in $[Q]^2$, she plays as described above in the lists before Claim~1. 

We have described Hider's strategy. It remains to check three cases based on which edges are white when the game terminates after infinitely many turns.

\textbf{Case 1.} There is a white edge in $[Q_0]^2$. (Analogous argument works for $Q_1$.)

On the one hand, we claim that the white edges can be colored so that the resulting coloring satisfies $\cals_0$, that is, $\calp_\infty\cap\cals_0\neq\emptyset$. Indeed, color every white edge between elements of $Q$ red; color the white edges between $A$ and $X$ according to $h$; color the remaining white edges according to Figure~\ref{f.separating_notions} (or more formally, according to the definition of $\cals_0$). Now properties (a)-(j) hold by the definition of Hider's strategy.

On the other hand, we also claim that the white edges can be colored so that $\cals_0$ fails for the resulting coloring, that is, $\calp_\infty\cap(2^\pairs\setminus\cals_0)\neq\emptyset$. More specifically, we claim that if we modify the coloring defined in the preceding paragraph only by coloring one white edge in $[Q_0]^2$ green, then $\cals_0$ fails. Suppose that $\cals_0$ holds. Based on properties (a)-(j), the figure, the description of Hider's strategy, and the preceding paragraph, we can check exactly as in the proof of Claim 1 that the vertex labels must reflect properties (a)-(i). But then property (j) fails because there are green edges in both $[Q_0]^2$ and $[Q_1]^2$.

Thus Seeker cannot decide whether $\cals_0$ holds without playing every edge in $[Q]^2$.

\textbf{Case 2.} There is a white edge $x_ia$ between $A$ and $X$.

Exactly as in the first part of Case~1, we can color the white edges so that $\cals_0$ holds. On the other hand, we claim that if we modify this coloring only by changing the color of $x_ia$, then $\cals_0$ fails. Suppose that $\cals_0$ holds. As in Case~1, the vertex labels reflect properties (a)-(i). Recall that $f(x)\neq f(y)$ whenever $|\supp x\Delta \supp y|=1$, hence, by property (j) and the fact that there is a green edge in $[Q]^2$, we conclude that $\cals_0$ fails if we modify the color of $x_ia$. Thus Seeker cannot decide whether $\cals_0$ holds without playing every edge between $A$ and $X$.

\textbf{Case 3.} There is a white edge $e$ that is not between $A$ and $X$, not in $[Q_0]^2$ and not in $[Q_1]^2$. Again, as in the first part of Case~1, we can color the white edges so that $\cals_0$ holds.

The failure of $\cals_0$ requires some additional work because to check that $\cals_0$ does not become true accidentally if we change the color of $e$, we need more than the rigidity of graphs with property $\cals_0$ that we have already used several times.

(I) Let $[V]^2=G'\cup R'$ be a partition of $[V]^2$ into green and red color classes that extend the color classes in Figure~\ref{f.separating_notions}. That is, if an edge is green (resp.~red) in the figure, then it belongs to $G'$ (resp.~$R'$). (A more formal definition that avoids the figure: if an edge occurs in the first (resp.~second) list at the beginning of the description of Hider's strategy, then it belongs to $G'$ (resp.~$R'$).)

(II) Let $[V]^2 = G'' \cup R''$ be another partition into green and red color classes, satisfying the same conditions except for one edge: there is exactly one edge $e \in [V]^2$ that appears as green or red in Figure~\ref{f.separating_notions} (or equivalently, in the two aforementioned lists), but in the partition $G'' \cup R''$, it belongs to the opposite color class.

\begin{lemma}\label{l.nonisomorphic_extensions}
Assume also that $\deg_{G'}(a)=\deg_{R'}(a)=\infty$. Then the graphs $(V,G')$ and $(V,G'')$ are not isomorphic.
\end{lemma}

\begin{proof}
Suppose that $\Phi:V\to V$ is an isomorphism between $(V,G')$ and $(V,G'')$. First, note that $P=\{x\in V:\ \deg_{R'}(x)<\infty\}$. (We slightly abuse Notation~\ref{n.notation_list} (a) since now the vertex set is $V$.) On the other hand, $\{x\in V:\ \deg_{R''}(x)<\infty\}$ is equal to either $P$ or $P\cup A$. Since $P$ is finite and $\Phi$ is an isomorphism, the latter is not possible. Thus $\Phi(P)=P$.
Second, note that $Q$ can be described as $\{x\in V:\ \exists y\in P\ (xy\in R')\}$. Also, $|\Phi(Q)|=|Q|=6$ since $\Phi$ is an isomorphism. Then $\Phi(P)=P$, $|\Phi(Q)|=|Q|=6$ and $\Phi(Q)\neq Q$ together would imply that the two colorings differ on at least two edges, hence we get $\Phi(Q)=Q$. Third, observe that $a$ is the unique vertex $x\in V$ with $\deg_{G'}(x)=\deg_{R'}(x)=\infty$. Since $a$ is also the only vertex for which $\deg_{G''}(x)=\deg_{R''}(x)=\infty$ \textit{can} hold, we must have $\Phi(a)=a$. Thus the classes $A$, $P$, $Q$, and $X$ are invariant under $\Phi$. In particular, $e$ cannot belong to $[A\cup P\cup Q]^2$ since $A\cup P\cup Q$ is finite.

\textbf{Subcase 1.} The edge $e$ is of the form $x_ix_{i+1}$. Then $\Phi$ maps the connected graph $G'|_X$ onto the disconnected one $G''|_X$, a contradiction.

\textbf{Subcase 2.} The edge $e$ is of the form $x_ix_j$ with $|i-j|>1$. Then $\Phi$ maps the tree $G'|_X$ onto the non-tree $G''|_X$, a contradiction.

\textbf{Subcase 3.} The edge $e$ joins vertices from $X$ and $P$. Then $\Phi$ maps a green edge between $X$ and $P$ to a red edge between $X$ and $P$, a contradiction.

\textbf{Subcase 4.} The edge $e$ joins vertices from $X$ and $Q$. Then $\Phi$ maps a red edge between $X$ and $Q$ to a green edge between $X$ and $Q$, a contradiction.

The proof of Lemma~\ref{l.nonisomorphic_extensions} is complete.
\end{proof}

Now we finish the proof of Case~3. We color the white edges so that the resulting color classes $G''$ and $R''$ satisfy the conditions of (II). In other words, the resulting coloring extends that of Figure~\ref{f.separating_notions} except for exactly one edge.

\textbf{Claim 2.} Property $\cals_0$ fails for $G''$.

It is clear from the definition of $\cals_0$ that any graph with property $\cals_0$ is isomorphic to a graph $G'$ that satisfies the conditions of (I) (that is, $G'$ and $R'=\pairs\setminus G'$ extend the coloring in Figure~\ref{f.separating_notions}) and for which $\deg_{G'}(a)=\deg_{R'}(a)=\infty$. By Lemma~\ref{l.nonisomorphic_extensions}, such a $G'$ cannot be isomorphic to $G''$, hence $\cals_0$ fails for $G''$, which proves the claim.

The proof of Theorem~\ref{t.separating_notions} is complete.
\end{proof}

We believe the following is worth recording. (See also Questions~\ref{q.separating_by_natural} and \ref{q.seperating_by_borel}.)

\begin{prop}\label{p.S_0_not_borel}
The property $\cals_0$ defined in the proof of Theorem~\ref{t.separating_notions} is not Borel (regardless of the choice of the map $f$ in Lemma~\ref{l.pattern}).
\end{prop}
\begin{proof}
(Sketch) First, observe that the bipartite graph $G$ defined in the proof of Lemma~\ref{l.pattern} does not have a Borel $2$-coloring because otherwise one of the color classes would be comeager in a basic clopen set of the form $N_s$, which contradicts the fact that the map $h: N_s\to N_s$ defined by $h(x)(j)=x(j)\iff j\neq l(s)$ is a homeomorphism of $N_s$ that swaps the color classes.

Second, we define a Borel map $\varphi:Y\to 2^\pairs$ such that for any coloring $f:Y\to\{0,1\}$ of $G$ and the associated property $\cals_0(f)$ (defined as in the proof of Theorem~\ref{t.separating_notions}), we have that $x\in f^{-1}(1)\iff \varphi(x)\in\cals_0(f)$, which shows that $\cals_0(f)$ cannot be Borel. For notational simplicity, let us identify $\nat$ with $V$ as in part (2) of the proof of Theorem~\ref{t.separating_notions}, so we have $\pairs=[V]^2$. We define $\varphi$ so that the value $\varphi(x)(e)$ depends on $x$ if and only if $e$ is of the form $x_ia$. For these edges, we put $\varphi(x)(x_ia)=1\iff x(i)=1$. For any $i\in\{0,1\}$, we put $\varphi(x)(e)=1-i$ if $e$ is an edge between vertices of $Q_i$. For the remaining edges $e$, we put $\varphi(x)(e)=1$ if and only if they are green in Figure~\ref{f.separating_notions} (or alternatively, they appear in the first list at the beginning of part (2) in the proof of Theorem~\ref{t.separating_notions}). It is clear from the definitions that $\varphi$ is Borel and $\varphi(x)\in\cals_0(f)$ whenever $x\in f^{-1}(1)$. To check that $\varphi(x)\notin\cals_0(f)$ if $x\notin f^{-1}(1)$, it suffices to verify that the only automorphism of $\varphi(x)$ is the identity. This can be done by a similar argument as in the proof of Claim~1 in the proof of Theorem~\ref{t.separating_notions}, which we leave to the reader.
\end{proof}

\section{Determinacy and descriptive complexity}\label{s.determinacy_and_descriptive_complexity}

In this section, we lay the foundations of studying $\infty$-elusiveness via descriptive set theory, and present partial results related to the following fundamental question, which, to the best of our knowledge, remains open.

\begin{question}\label{q.undetermined}
Is every game of the form $G(\cals)$ determined? In other words, is every $\infty$-elusive graph property strongly $\infty$-elusive?
\end{question}

Recall from Section~\ref{s.prelim} that we view graph properties as isomorphism-invariant subsets of the Cantor space $2^\pairs$, therefore it makes sense to talk about the descriptive complexity of a graph property. We follow the usual descriptive-set-theoretic formalism using trees to describe infinite games, see Subsection~2.3.

Let $A=\pairs\cup\{\text{green, red}\}$.
Recall that $e_0,e_1,\ldots$ is an enumeration of $\pairs$. For any graph property $\cals$, the rules of the game $G(\cals)$ define a subtree $F$ of $A^{<\omega}$ in a natural way. We put $s\in F$ if and only if the following hold:
\begin{itemize}
    \item $s(2i+1)\in\{\text{green, red}\}$ for every $i\in\nat$,
    \item $s(2i)\in\pairs$ for every $i\in\nat$,
    \item $s(2i)\neq s(2j)$ if $i,j\in\nat$ and $i\neq j$.
\end{itemize}

Note that $F$ does not depend on $\cals$. To every $x\in [F]$ we associate the color classes
$$G^x=\{e\in\pairs:\ \exists i\in\nat\ (x(2i)=e\land x(2i+1)=\text{green})\},$$
$$R^x=\{e\in\pairs:\ \exists i\in\nat\ (x(2i)=e\land x(2i+1)=\text{red})\},$$
$$W^x=\pairs\setminus(G^x\cup R^x).$$
Let also $\calp^x=\{G\subseteq\pairs:\ G^x\subseteq G\subseteq G^x\cup W^x\}$.

\begin{remark}\label{r.open_inverse_images}
For any $e\in\pairs$ the sets $\{x\in [F]:\ e\in G^x\}$ and $\{x\in [F]: e\in R^x\}$ are open.
\end{remark}

In harmony with the definitions in Section~\ref{s.prelim}, the \textbf{winning set} of Seeker is
$$A_\cals=\{x\in [F]:\ (W^x\neq\emptyset)\land((\calp^x\subseteq \cals)\lor(\calp^x\subseteq 2^\pairs\setminus\cals))\}.$$
The (descriptive) complexity of the game $G(\cals)$ is the complexity of $A_\cals$ as a subset of $[F]$. First, we analyze the complexity of $G(\cals)$ when $\cals$ is Borel.

\begin{prop}\label{p.borel_induces_coanalytic}
If $\cals$ is a Borel graph property, then $G(\cals)$ is a co-analytic game.
\end{prop}

\begin{proof}
By definition, $A_\cals$ is of the form $\cala\cap(\calb\cup\calc)$, where
$$\cala=\{x\in [F]:\ \exists e\in\pairs\ \forall i\in\nat\ (x(2i)\neq e)\},$$
$$\calb=\{x\in [F]:\ \forall G\in 2^\pairs\ ((G\in\cals)\lor (G\notin\calp^x))\},$$
$$\calc=\{x\in [F]:\ \forall G\in 2^\pairs\ ((G\notin\cals)\lor (G\notin \calp^x))\}.$$
Let us unravel the definitions:
$$G\in\calp^x\iff ((G^x\subseteq G)\land (G\cap R^x=\emptyset))\iff$$
$$\iff(\forall e\in\pairs\ ((e\in G\lor e\notin G^x)\land (e\notin G\lor e\notin R^x))),$$
where
$$e\in G^x\iff \exists i\in\nat\ ((x(2i)=e)\land (x(2i+1)=\text{green}))$$
and
$$e\in R^x\iff\exists i\in\nat\ ((x(2i)=e)\land (x(2i+1)=\text{red})),$$
which shows that $\{(x,G)\in[F]\times 2^\pairs:\ G\in\calp^x\}$ is Borel. This makes $\calb$ and $\calc$ co-analytic, hence $A_\cals$ is also co-analytic.
\end{proof}

The following theorem shows that Proposition~\ref{p.borel_induces_coanalytic} is sharp. Thus the Borel Determinacy Theorem does not answer Question~\ref{q.undetermined}, not even for Borel graph properties.

\begin{theorem}\label{t.complete_coanalytic}
Let $\cals$ be the property of being an infinite path. Then $\cals$ is Borel and $G(\cals)$ is a complete co-analytic game.
\end{theorem}

\begin{proof}
Observe that
$$G\in\cals\iff \exists n\in\nat\ ((\deg_G(n)=1)\land\forall m\in\nat\setminus\{n\}\ (\deg_G(m)=2))\land G\text{ is connected}.$$
Also, for any $k\in\nat$, we can write $\deg_G(n)=k$ as
$$\exists \{m_0,\ldots,m_{k-1}\}\in[\nat]^k\ (\forall i<k\ (nm_i\in G)\land\forall m\in\nat\setminus\{m_0,\ldots,m_{k-1}\}\ (nm\notin G)),$$
and $G$ is connected if and only if
$$\forall \{n,m\}\in\pairs\ \exists k\geq 2\ \exists i_0,\ldots,i_{k-1}\in\nat\ (i_0=n\land i_{k-1}=m\land \forall j<k\ (i_ji_{j+1}\in G)).$$
Thus $\cals$ is Borel. Now we construct a continuous reduction of the complete co-analytic set $\wellfounded\setminus\{\emptyset,\{\emptyset\}\}$ (see Subsection~2.3) to $A_\cals$. For notational convenience, let us fix an indexing of natural numbers with finite sequences: $\nat=\{a_s:\ s\in\nat^{<\omega}\}$. Also, for every $s\in\nat^{<\omega}$ and $i\in \nat$, let $s_i=s$ and $e^s=e_i$ if and only if $s$ is of the form $t^\frown j$ for some $j\in\nat$, $t\in\nat^{<\omega}$ such that $a_ta_s=e_i$. Let $E_T=\{e^s:\ s\in T\}$. Thus $E_T$ is isomorphic to $T$ if $T$ is viewed as a graph-theoretic tree. The map $\Phi:\trees\to[F]$ is defined as follows. We assign to every $T\in\trees$ the unique sequence $x_T$ such that $x_T(2i+1)=\text{red}$ for every $i\in\nat$ and $(x_T(2i))_{i\in\nat}$ is an increasing enumeration of the set $\pairs\setminus E_T$ in the sense that $f:\nat\to\nat$ defined by $f(i)=j\iff x_T(2i)=e_j$ is order-preserving. Thus $x_T$ is a run of the game $G(\cals)$ that results in the coloring that is white on $E_T$ and red on its complement.

\textbf{Claim.} The map $\Phi$ is continuous.

For any $n\in\nat^+$ and $T\in\trees$, let $e_{i_{n,T}}$ be the greatest-indexed edge in $\{x_T(2k):\ k<n\}$. We claim that for any $T'\in\trees$, if $(s_i\in T\iff s_i\in T')$ holds whenever $i\leq i_{n,T}$ and $s_i$ is defined, then $x_T|_{2n}=x_{T'}|_{2n}$. Otherwise $x_T|_{2n}\neq x_{T'}|_{2n}$ would be witnessed by an edge $e_j\in E_T\Delta E_{T'}$ with $j<i_{n,T}$, which contradicts the assumption $s_j\in T\iff s_j\in T'$ and thereby proves the claim.

It remains to verify that $T\in \wellfounded\setminus\{\emptyset,\{\emptyset\}\}\iff x_T\in A_\cals$.

First, note that for any $T\in\trees$ we have $G^{x_T}=\emptyset$ and $W^{x_T}=E_T$. In particular, $W^{x_T}=\emptyset\iff T\in\{\emptyset,\{\emptyset\}\}$. Thus, it suffices to prove that $E_T$ contains an infinite path if and only if $T\notin\wellfounded$. If $T\notin\wellfounded$, then $E_T$ clearly contains an infinite path. If $T\in\wellfounded$, then it contains no infinite path since any infinite path in $E_T$ would go along a branch of $T$ after finitely many steps.
\end{proof}

Although this is irrelevant from the perspective of Theorem~\ref{t.complete_coanalytic}, the property of being an infinite path is $\infty$-elusive since it is sensitive (recall Remarks~\ref{r.monotone_sensitive} and \ref{r.sensitive_elusive}).

Observe that for many graph properties $\cals$, either $\cals$ or $2^\pairs\setminus\cals$ is monotone (recall Definition~\ref{d.monotone}). The same holds for most properties considered in this paper. Thus, in contrast to Theorem~\ref{t.complete_coanalytic}, the following theorem says that for a wide range of natural graph properties $\cals$, the game $G(\cals)$ is determined.

\begin{theorem}\label{t.monotone_borel}
If $\cals$ is a monotone Borel graph property, then $G(\cals)$ is a Borel game. In particular, it is determined.
\end{theorem}

\begin{proof}
Recall that the winning set of Seeker is
$$A_\cals=\{x\in [F]:\ (W^x\neq\emptyset)\land((\calp^x\subseteq \cals)\lor(\calp^x\subseteq 2^\pairs\setminus\cals))\}.$$
Note that if $\cals$ is monotone, then $A_\cals$ can be written as
$$\{x\in [F]:\ (W^x\neq\emptyset)\land ((G^x\in\cals)\lor (G^x\cup W^x\notin\cals))\}.$$
To see that this set is Borel it suffices to check that the following $[F]\to 2^\pairs$ maps are $\mathbf{\Sigma^0_2}$-measurable (which means that the inverse image of any open set is $F_\sigma$):
$$x\mapsto G^x,\quad x\mapsto W^x,\quad x\mapsto G^x\cup W^x.$$ 
Since the family of $F_\sigma$ sets is closed under finite intersections and countable unions, it suffices to check that for any $e\in\pairs$ the following sets are $F_\sigma$:
$$\{x\in [F]:\ e\in G^x\},\ \{x\in [F]:\ e\notin G^x\},\ \{x\in [F]:\ e\in W^x\},$$
$$\{x\in [F]:\ e\notin W^x\},\ \{x\in [F]:\ e\in G^x\cup W^x\},\ \{x\in [F]:\ e\notin G^x\cup W^x\}.$$
It is clear from Remark~\ref{r.open_inverse_images} that these sets are open, closed, closed, open, closed, and open, respectively.
\end{proof}

Finally, we show that if the complexity of $\cals$ is sufficiently low, then we do not need the assumption of monotonicity to conclude that $G(\cals)$ is determined.

\begin{theorem}\label{t.ambiguous}
If $\cals$ is a $\mathbf{\Delta^0_2}$ graph property, then $G(\cals)$ is a $\mathbf{\Pi^0_3}$ game. In particular, it is determined.
\end{theorem}

\begin{proof}
Note that for any $x\in [F]$ the set $\calp^x=\{G\subseteq\pairs:\ G^x\subseteq G\subseteq G^x\cup W^x\}$ is compact in $2^\pairs$. 

\textbf{Claim.} The map $p:[F]\to \calk(2^\pairs)$, $x\mapsto \calp^x$ is $\mathbf{\Sigma^0_2}$-measurable, where $\calk(2^\pairs)$ is the hyperspace of $2^\pairs$ (see Section~\ref{s.prelim} for the definitions).

Indeed, by Proposition~\ref{p.vietoris_subbasis}, it suffices to check that for any $s\in 2^{<\omega}$ the sets
$$\{x\in [F]:\ \calp^x\cap N_s\neq\emptyset\}\quad\text{and}\quad\{x\in[F]:\ \calp^x\cap N_s=\emptyset\}$$
are $F_\sigma$, where $N_s=\{G\in 2^\pairs:\ \forall i<l(s) (e_i\in G\iff s(i)=1)\}$. By the definitions, we have $\calp^x\cap N_s\neq\emptyset$ if and only if both $(\forall i<l(s) (s(i)=1\implies e_i\notin R^x))$ and $(\forall i<l(s) (s(i)=0\implies e_i\notin G^x))$ holds. Thus, it follows from Remark~\ref{r.open_inverse_images} that $\{x\in [F]:\ \calp^x\cap N_s\neq\emptyset\}$ is closed and $\{x\in[F]:\ \calp^x\cap N_s=\emptyset\}$ is open, which proves the claim.

By definition,
$$A_\cals=\{x\in [F]:\ (W^x\neq\emptyset)\land((\calp^x\subseteq \cals)\lor(\calp^x\subseteq 2^\pairs\setminus\cals))\},$$
hence it suffices to verify that
$$A_0=\{x\in[F]:\ W^x\neq\emptyset\},\ \ A_1=\{x\in[F]:\ \calp^x\subseteq\cals\},\ \ \text{ and }\ \ A_2=\{x\in[F]:\ \calp^x\subseteq 2^\pairs\setminus\cals\}$$ are $\mathbf{\Pi^0_3}$ since the family of $\mathbf{\Pi^0_3}$ sets is closed under finite unions and countable intersections. First,
$$A_0=\{x\in[F]:\ \exists n\in\nat\ \forall k\in\nat\ (x(2k)\neq e_n)\}$$
is clearly $F_\sigma$. Second, $\cals$ is $\mathbf{\Delta^0_2}$ by assumption, therefore we can write $\cals=\bigcap_{i\in\nat} \calu_i$ and $2^\pairs\setminus\cals=\bigcap_{j\in\nat}\calv_j$ with $\calu_i, \calv_j\subseteq 2^\pairs$ open for all $i,j\in\nat$. Then it follows from the claim and the definition of the Vietoris topology that 
$$A_1=\{x\in[F]:\ \forall i\in\nat\ (\calp^x\subseteq U_i)\}$$
and
$$A_2=\{x\in[F]:\forall j\in\nat\ (\calp^x\subseteq V_j)\}$$
are $\mathbf{\Pi^0_3}$, which concludes the proof.
\end{proof}

\section{Open problems}\label{s.open_problems}

First, let us recall the main question of Section~\ref{s.determinacy_and_descriptive_complexity}.

\begin{question}
Is every game of the form $G(\cals)$ determined?
\end{question}

The following variant is also natural.

\begin{question}
Is $G(\cals)$ determined for every monotone property $\cals$?
\end{question}

Regarding the relation between the transfinite game and the infinite but not transfinite game, Theorem~\ref{t.separating_notions} proves the existence of a graph property that is $\infty$-elusive but not elusive. However, the constructed property $\cals_0$ is rather tailor-made and, by Proposition~\ref{p.S_0_not_borel}, it is not even Borel. This raises two questions.

\begin{question}\label{q.separating_by_natural}
Is there a \emph{natural} graph property $\cals$ that is $\infty$-elusive but not elusive? By natural graph property, we mean a property a graph theorist would normally consider.
\end{question}

\begin{question}\label{q.seperating_by_borel}
Is there a \emph{Borel} graph property $\cals$ that is $\infty$-elusive but not elusive?
\end{question}

We conclude the paper by asking whether the natural generalization of Theorem~\ref{t.isolated} holds.

\begin{question}
Is there $k\geq 2$ such that the property ``every vertex has degree at least $k$'' is $\infty$-elusive?  
\end{question}

Csernák and Soukup have proven that this property is not elusive for every $k\geq 1$ but it remains unclear whether Seeker has a winning strategy if they are not allowed to play transfinitely.

\section*{Acknowledgment}

This research was supported by the National Research, Development and Innovation Office -- NKFIH, grant no. 124749 and 146922.
The first author was also supported by the National Research, Development and Innovation Office -- NKFIH, grant no.~129211. The second and third authors were also supported by the Hungarian Academy of Sciences Momentum Grant no. 2022-58.

\printbibliography[title={References}]

\section{Appendix}

Here we finish the proof of Theorem~\ref{t.bipartite} by checking 21 cases. Recall the following figure and conditions (I)-(III).

\begin{figure}[ht]
\centering
\includegraphics[scale=0.5]{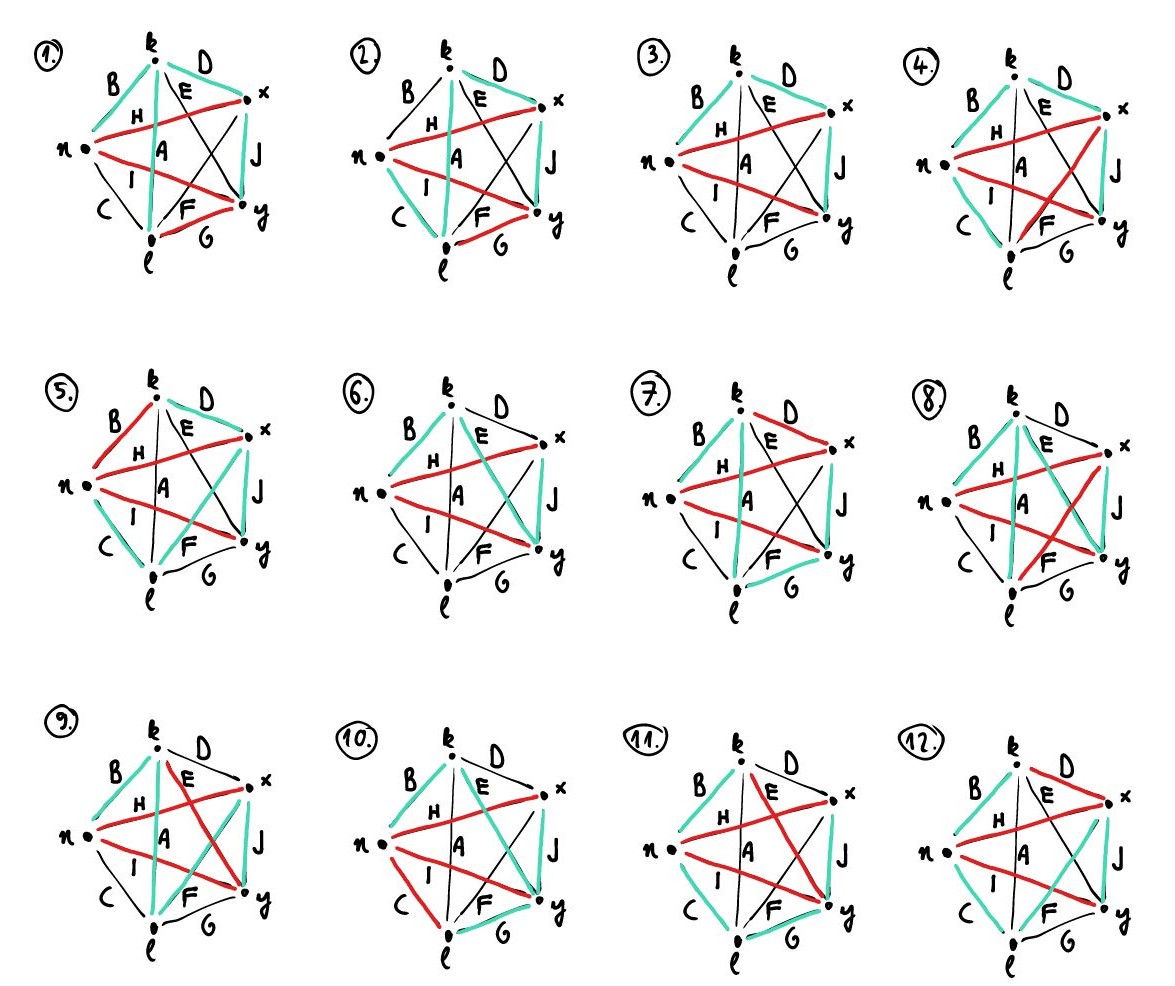}
\end{figure}

(I) There is a green cycle of odd length. 

(II) There is no green-white cycle of odd length. 

(III) The green edges form a connected bipartite graph that covers $n$, and there is no white edge between the vertices of this graph.

As long as Seeker does not play $A$, $B$, $C$, or the last white edge from $\{D, G\}$ or $\{E, F\}$, Hider responds red, and $\lnot$(II) remains true.

\textbf{Case 1.} Seeker plays the last edge from $\{D,G\}$ or $\{E,F\}$ before playing $A$, $B$, or $C$. By symmetry, we may assume that this last edge is $D$. Hider colors it green, $\lnot$(I) remains true. From now on as long as Seeker plays edges only from $\{E, F\}$, Hider responds red, and $\lnot$(II) remains true. At some point, Seeker must play $A$, $B$, or $C$.

\textbf{Subcase 1.1.} Seeker plays $A$ first from $\{A, B, C\}$. Hider responds green, $\lnot$(I) remains true. From now on as long as Seeker plays edges only from $\{E, F\}$, Hider responds red, $\lnot$(II) remains true. At some point, Seeker must play an edge from $\{B, C\}$, then Hider responds green thereby reaching Position 1 or 2.

\textbf{Subcase 1.2.} Seeker plays $B$ first from $\{A, B, C\}$. Hider responds green thereby reaching Position 3.

\textbf{Subcase 1.3.} Seeker plays $C$ first from $\{A, B, C\}$. Hider responds green, $\lnot$(I) remains true. From now on as long as Seeker plays edges only from $\{E, F\}$, Hider responds red, $\lnot$(II) remains true. At some point, Seeker must play $A$ or $B$.

\textbf{Subcase 1.3.1.} Seeker plays $A$ first from $\{A, B\}$. Hider responds green thereby reaching Position 2.

\textbf{Subcase 1.3.2.} Seeker plays $B$ first from $\{A, B\}$. Now if $F$ is red, then Hider responds green thereby reaching Position 4. If $F$ is white, then Hider responds red, $\lnot$(II) remains true ($DAF$ is a green-white triangle). From now on as long as Seeker does not play $A$ or $F$, Hider responds red, $\lnot$(II) holds. When Seeker plays $A$ or $F$, Hider responds green thereby reaching Position 2 or 5.

\textbf{Case 2.} Seeker plays $B$ or $C$ before playing $A$ or the last edge from $\{D, G\}$ or $\{E, F\}$. By symmetry, we may assume that she plays $B$. Hider responds green, $\lnot$(I) remains true.

\textbf{Subcase 2.1.} At this point, both $D$ and $E$ are white. Then $DEJ$ is a green-white triangle. From now on as long as Seeker does not play $D$ or $E$, Hider responds red, $\lnot$(II) holds. When Seeker plays $D$ or $E$, Hider responds green thereby reaching Position 3 or 6.

\textbf{Subcase 2.2.} At this point, exactly one of $D$ and $E$ is white. By symmetry, we may assume that it is $E$. Then $G$ is also white.

\textbf{Subcase 2.2.1.} Seeker plays $A$ next. Hider responds green, $\lnot$(I) remains true. From now on as long as Seeker does not play $G$ or the last edge from $\{E, F\}$, Hider responds red, $\lnot$(II) remains true (either $AGE$ or $FGJ$ is a green-white triangle). When Seeker plays $G$ or the last edge from $\{E, F\}$, Hider responds green thereby reaching Position 7, 8, or 9.

\textbf{Subcase 2.2.2.} Seeker plays $C$ next. Hider responds red, $\lnot$(II) remains true.

\textbf{Subcase 2.2.2.1.} Seeker plays $A$ next. Hider responds green and follows Subcase 2.2.1.

\textbf{Subcase 2.2.2.2.} Seeker plays $E$ next. Hider responds green thereby reaching Position 6.

\textbf{Subcase 2.2.2.3.} Seeker plays $F$ next. Hider responds red, $\lnot$(II) remains true. From now on as long as the graph formed by the green edges is not connected, Hider always responds green, $\lnot$(I) remains true. When it becomes connected, they reach Position 6, 7, 8, or 10.

\textbf{Subcase 2.2.2.4.} Seeker plays $G$ next. Hider responds green, $\lnot$(I) remains true. From now on independently of the order in which Seeker plays the edges, Hider responds red to $F$ and green to the first edge from $\{A, E\}$, $\lnot$(II) remains true. When the first edge from $\{A, E\}$ is colored green, they reach Position 7 or 10.

\textbf{Subcase 2.2.3.} Seeker plays $E$ next. Hider responds green thereby reaching Position 6.

\textbf{Subcase 2.2.4.} Seeker plays $F$ next. Hider responds red, $\lnot$(II) remains true.

\textbf{Subcase 2.2.4.1.} Seeker plays $A$ next. Hider responds green, $\lnot$(I) remains true. From now on independently of the order in which Seeker plays the edges, Hider responds red to $C$ and green to the first edge from $\{E,G\}$, $\lnot$(II) remains true. When the first edge from $\{E, G\}$ is colored green, they reach Position 7 or 8.

\textbf{Subcase 2.2.4.2.} Seeker plays $C$ next. Hider responds red and follows Subcase 2.2.2.3.

\textbf{Subcase 2.2.4.3.} Seeker plays $E$ next. Hider responds green thereby reaching Position 6.

\textbf{Subcase 2.2.4.4.} Seeker plays $G$ next. Hider responds green, $\lnot$(I) holds. From now on independently of the order in which Seeker plays the edges, Hider responds green only to $A$ or the last edge from $\{C, E\}$, and $\lnot$(II) remains true (either $ABC$ or $AGE$ is a green-white triangle). When $A$ or the last edge from $\{C, E\}$ is colored green, they reach Position 7, 10, or 11.

\textbf{Subcase 2.2.5.} Seeker plays $G$ next. Hider responds green, $\lnot$(I) remains true.

\textbf{Subcase 2.2.5.1.} Seeker plays $A$ next. Hider responds green thereby reaching Position 7.

\textbf{Subcase 2.2.5.2.} Seeker plays $C$ next. Hider responds red and Hider follows Subcase 2.2.2.4.

\textbf{Subcase 2.2.5.3.} Seeker plays $E$ next. Hider responds red, $\lnot$(II) remains true. From now on independently of the order in which Seeker plays the edges, Hider responds red to $F$ and green to the first edge from $\{A, C\}$, and $\lnot$(II) remains true. When the first edge from $\{A, C\}$ is colored green, they reach Position 7 or 11.

\textbf{Subcase 2.2.5.4.} Seeker plays $F$ next. Hider responds red and follows Subcase 2.2.4.4.

\textbf{Subcase 2.3.} At this point, both $D$ and $E$ are red. From now on independently of the order in which Seeker plays the edges, Hider responds green exactly to the first edge from $\{A, C\}$ and the first edge from $\{F, G\}$, meanwhile $\lnot$(I) and $\lnot(II)$ remain true. When both $\{A, C\}$ and $\{F, G\}$ contain a green edge, they reach Position 7, 9, 11, or 12.

\textbf{Subcase 3.} Seeker plays $A$ before playing $B$, $C$ or the last edge from $\{D, G\}$ or $\{E,F\}$. Hider responds green, $\lnot(I)$ remains true. Then independently of the order in which Seeker plays the edges, Hider plays as follows. She responds green to the first and red to the second edge from $\{B, C\}$. As long as Seeker does not play the last edge from $\{D, G\}$ or $\{E, F\}$, Hider responds red to every edge from $\{D, E, F, G\}$. In the first turn when Seeker plays the last edge from $\{D, G\}$ or $\{E, F\}$, Hider responds green. After that, she responds red to every edge from $\{D, E, F, G\}$. As one checks easily, $\lnot$(I) remains true. Also, $\lnot$(II) remains true --- since at least one of the triangles $ABC$, $AFD$, $AGE$, $DEJ$, and $FGJ$ is green-white --- until they reach Position 1, 7, 8, 9, or one of their reflections to the horizontal axis, which is also a winning position by symmetry.

The proof of Theorem~\ref{t.bipartite} is complete.

\end{document}